\documentclass[a4paper]{article}

\usepackage{amsmath,amssymb,color}

\usepackage{graphicx}
 \usepackage{subfigure}
\DeclareGraphicsExtensions{.pdf,.eps}

\newcommand{\bu}{\boldsymbol u}
\newcommand{\Om}{\Omega}

\newcommand{\bv}{\boldsymbol v}
\newcommand{\bV}{\boldsymbol V}
\newcommand{\bw}{\boldsymbol w}
\newcommand{\bbeta}{\boldsymbol  \eta}
\newcommand{\bzeta}{\boldsymbol \zeta}

\newcommand{\btau}{\boldsymbol \tau}

\newcommand{\be}{\boldsymbol e}

\newcommand{\bvar}{\boldsymbol \varphi}

\newcommand{\bs}{\boldsymbol s}

\newcommand{\bff}{\boldsymbol f}

\newtheorem{Theorem}{Theorem}[section]

\newtheorem{remark}[Theorem]{Remark}

\newtheorem{Proof}{{\em Proof:}}

\newenvironment{proof}{\begin{Proof}\rm}{\hfill $\Box$ \end{Proof}}
\usepackage{hyperref}

\usepackage{macros}

\title{Error analysis of proper orthogonal decomposition stabilized methods for incompressible flows}

\author{  Julia Novo\thanks{Departamento de
Matem\'aticas, Universidad Aut\'onoma de Madrid, Spain.  Research is supported
by Spanish MINECO
under grant MTM2016-78995-P (AEI/FEDER, UE) cofinanced by FEDER funds (julia.novo@uam.es).}
\and Samuele Rubino\thanks{Department EDAN \& IMUS, Universidad de Sevilla, Spain. Research is supported by
Spanish MCINYU under grant RTI2018-093521-B-C31 and Spanish State Research Agency
through the national programme Juan de la Cierva-Incorporaci\'on 2017 (samuele@us.es).}}

\begin{document}

\maketitle

\begin{abstract}
 Proper orthogonal decomposition (POD) stabilized methods for the Navier-Stokes equations are considered and analyzed.
 We consider two cases, the case in which the snapshots are based on a non inf-sup stable method and the case in which the
 snapshots are based on an inf-sup stable method. For both cases we construct approximations to the velocity and the pressure.
 For the first case, we analyze a method in which the snapshots are based on a stabilized scheme with equal order polynomials for the velocity and the pressure with Local Projection Stabilization (LPS) for the gradient of the velocity and the pressure. For the POD method we add the same kind of LPS stabilization for the gradient of the velocity  and the pressure than the direct method, together with grad-div stabilization.
  In the second case, the snapshots are based on an inf-sup stable Galerkin method with grad-div stabilization and for the POD model we apply also grad-div stabilization. In this case, since the snapshots are discretely divergence-free, the pressure can be removed from the formulation of the POD approximation to the velocity. To approximate the pressure, needed in many engineering applications, we use a supremizer pressure recovery method.
  Error bounds with constants independent on inverse powers of the viscosity parameter are proved for both methods. Numerical experiments show the accuracy and performance of the schemes.
\end{abstract}

76B75. 

\section{Introduction}

Reduced order models (ROM) are a fairly extensive technique applied in many different fields to reduce the computational cost of direct
numerical simulations while keeping enough accurate numerical approximations. In particular, ROM have been extensively applied in recent years to model incompressible flows \cite{Rozza15,IliescuJohn14,RebholzIliescu17,Quarteroni16,RozzaStabile18,Wang12}. Proper Orthogonal Decomposition (POD) method provides the elements (modes) of the reduced basis from a given database (snapshots) which are computed by means of a direct or full order method (FOM).

In this paper, we study the numerical approximation of incompressible flows with stabilized POD-ROM.
We consider the Navier-Stokes equations
\begin{align}
\label{NS} \partial_t\bu -\nu \Delta \bu + (\bu\cdot\nabla)\bu + \nabla p &= \bff &&\text{in }\ (0,T]\times\Omega,\nonumber\\
\nabla \cdot \bu &=0&&\text{in }\ (0,T]\times\Omega,
\end{align}
in a bounded domain $\Omega \subset {\mathbb R}^d$, $d \in \{2,3\}$ with initial condition $\bu(0)=\bu^0$. In~\eqref{NS},
$\bu$ is the velocity field, $p$ the kinematic pressure, $\nu>0$ the kinematic viscosity coefficient,
 and $\bff$ represents the accelerations due to external body forces acting
on the fluid. The Navier-Stokes equations \eqref{NS} must be complemented with boundary conditions. For simplicity,
we only consider homogeneous
Dirichlet boundary conditions $\bu = \boldsymbol 0$ on $\partial \Omega$.

In practice, for computing the snapshots one can use an inf-sup stable or a non inf-sup stable method. In this paper, we analyze two methods, one starting from a non inf-sup stable method and the other starting from an inf-sup stable method. For both cases we construct approximations to the velocity and the pressure since in many engineering problems not only approximations to the velocity are required, but also to the pressure (e.g., to compute forces on bodies in the flow and for incompressible shear flows \cite{NPM05}, as the mixing layer or the wake flow \cite{Tadmor11}, where neglecting the pressure may lead to large amplitude errors).

In \cite{samuele_pod} a POD stabilized method for the Navier-Stokes equations is introduced and analyzed. The method uses snapshots for the
velocity that do not satisfy a weakly discrete divergence-free condition since are based on equal order (say order $l$) velocity and pressure approximations. A LPS-type stabilization term for the pressure is introduced in the POD model. This term is inspired in the term-by-term LPS stabilization introduced in \cite{Cha_maca}, see also \cite{ahmed_etal}, \cite{nos_lps}. The stabilization term is based on a fluctuation operator of a locally stable projection onto the space of polynomials of degree $l-1$. The method in \cite{samuele_pod} allows to compute both approximations to the velocity and the pressure. Apart from LPS stabilization
for the pressure no other stabilization terms are considered in this method. A penalty term is added to the variational formulation to avoid the zero mean condition for the pressure. This technique introduces an extra parameter $\sigma$ that is assumed to be small (of order $10^{-6}$). In the error analysis of the
method there is a factor behaving as $\sigma$. To  handle a divergence error term appearing in the error analysis in \cite{samuele_pod} a saturation property
is applied. If we denote by $\alpha$ the saturation constant between the space of divergence POD velocity approximations and the space of POD pressure 
approximations then the bounds in \cite{samuele_pod} are multiplied by $\alpha^2/\sigma$. The experimental value of $\alpha$ in the numerical examples
in \cite{samuele_pod} has size around $10^{-2}$ so that $\alpha^2/\sigma\approx 10^2$. The error bounds in \cite{samuele_pod} are not independent
on inverse powers of the viscosity parameter.

In the first part of the present paper we consider the same method of \cite{samuele_pod} with some differences. First of all, we do not
add a penalty term for the pressure. Since both, the parameter $\sigma$ of the penalty term and its inverse $\sigma^{-1}$,  appear in the error analysis of \cite{samuele_pod} multiplying the
constants in the error bounds it is a problem to choose the optimal value for $\sigma$, while keeping an optimal order for the method.  Secondly, apart from term-by-term LPS stabilization for the pressure  we add LPS 
stabilization for the gradient of the velocity and also grad-div stabilization to the POD model. To compute the snapshots, we use a stabilized FOM. More precisely, the snapshots for the model are based on equal order velocity and pressure approximations of a stabilized finite element method that uses term-by-term LPS stabilization for the gradient of the velocity and the pressure. In \cite{nos_lps} error bounds for this method are proved in which the constants do not depend on inverse powers of the viscosity. An error bound for the $L^2$ error of the velocity of order $l+1/2$ is obtained (where $l$ is the local degree of the velocity and pressure approximations). Starting from an optimal order FOM allows to improve the accuracy of the model we present. As in \cite{samuele_pod}, the method we consider in this paper provides both approximations to the velocity and the pressure. In this paper, to handle a divergence error term appearing in the error analysis, we use some properties of the POD basis functions inherited from the direct model (the stabilized method), instead of a saturation assumption.   Then, we are able to handle all the terms in the error analysis using the same kind of stabilization in the POD model as in the FOM method. Moreover, we prove error bounds in which the constants do not depend on inverse powers of the viscosity parameter. To this end, the added grad-div term is essential.

To our knowledge this is the first time this kind of viscosity independent bounds are obtained for POD models. In \cite{pod_da_nos}  a POD data assimilation scheme for the Navier-Stokes equations is analyzed. The method
in \cite{pod_da_nos}  uses also grad-div stabilization. This method only provides approximations for the velocity. Some of the constants in the error bounds in \cite{pod_da_nos}  are also viscosity independent. Indeed, we follow some of the ideas in  \cite{pod_da_nos} for the error analysis of the present paper. However, in \cite{pod_da_nos} the snapshots are based on a non stabilized standard Galerkin method and, as a consequence, the error coming from the snapshots depends on inverse powers of the viscosity. Although a stabilized method could be used for the snapshots in \cite{pod_da_nos} to avoid this dependency, it could not be an equal order velocity-pressure method (as in the present paper) since the discrete divergence-free condition of the snapshots is strongly used in the error analysis of \cite{pod_da_nos}. To solve this problem,
we also introduce and analyze a second method in which the snapshots are computed with an inf-sup stable Galerkin method with grad-div stabilization for which error bounds with constants independent on the viscosity have been proved in \cite{NS_grad_div}. For the POD model we also add grad-div stabilization. In a first step, we compute only a POD approximation to the velocity, since the pressure can be removed from the formulation using the discrete divergence-free condition of the POD basis functions. In a second step, following \cite{schneier} we apply a supremizer \cite{Rozza15,RozzaVeroy07} pressure recovery method to get a POD pressure approximation. We get bounds both for the velocity and the pressure in which the constants do not depend on inverse powers of the viscosity.

The outline of the paper is as follows. In Section \ref{sec:PN} we introduce some preliminaries and notations. Section \ref{sec:POD} is devoted to recall the POD method and get some a priori bounds for the $L^2$-orthogonal projection of the FOM velocity onto the POD velocity space. In Section \ref{sec:LPS-ROM} we introduce and analyze the stabilized POD method based on a non inf-sup stable FOM. In Section \ref{sec:grad-div-ROM} we introduce and analyze the grad-div stabilized POD method based on an inf-sup stable FOM. In Section \ref{sec:num} some numerical experiments show the accuracy and performance of the methods. Finally, Section \ref{sec:Concl} presents the main conclusions.

\section{Preliminaries and notation}\label{sec:PN}
The following Sobolev embeddings \cite{Adams} will be used in the analysis: For
$q \in [1, \infty)$, there exists a constant $C=C(\Omega, q)$ such
that
\begin{equation}\label{sob1}
\|v\|_{L^{q'}} \le C \| v\|_{W^{s,q}}, \,\,\quad
\frac{1}{q'}
\ge \frac{1}{q}-\frac{s}{d}>0,\quad q<\infty, \quad v \in
W^{s,q}(\Omega)^{d}.
\end{equation}
The following inequality can be found in  \cite[Remark 3.35]{John}
\begin{eqnarray}\label{diver_vol}
\|\nabla \cdot \bv\|_0\le \|\nabla   \bv \|_0,\quad \bv\in H_0^1(\Omega)^d.
\end{eqnarray}
Let us denote by $Q=L_0^2(\Omega)=\left\{q\in L^2(\Omega)\mid (q,1)=0\right\}$.
Let $\mathcal{T}_{h}=(\tau_j^h,\phi_{j}^{h})_{j \in J_{h}}$, $h>0$ be a family of partitions of $\overline\Omega$, where $h$ denotes the maximum diameter of the elements $\tau_j^h\in \mathcal{T}_{h}$, and $\phi_j^h$ are the mappings from the reference simplex $\tau_0$ onto $\tau_j^h$.
We shall assume that the partitions are shape-regular and quasi-uniform. 
We define the following finite element spaces
\begin{eqnarray*}
Y_h^l&=& \left\{v_h\in C^0(\overline\Omega)\mid {v_h}_{\mid_K}\in {\Bbb P}_l(K),\quad \forall K\in \mathcal T_h\right\}, \ l\ge 1,\nonumber\\
{\boldsymbol Y}_h^l&=&(Y_h^l)^d,\quad {\boldsymbol X}_h^l={\boldsymbol Y}_h^l\cap H_0^1(\Omega)^d, \nonumber\\
Q_h^l&=&Y_h^l\cap L_0^2(\Omega).
\end{eqnarray*}
\begin{eqnarray}\label{eq:V}
{\boldsymbol V}_{h,l}={\boldsymbol X}_h^l\cap \left\{ {\boldsymbol \chi}_{h} \in H_0^1(\Omega)^d \mid
(q_{h}, \nabla\cdot{\boldsymbol\chi}_{h}) =0  \quad\forall q_{h} \in Q_{h}^{l-1}
\right\},\quad l\geq 2.
\end{eqnarray}If the family of
meshes is quasi-uniform then  the following inverse
inequality holds for each $\bv_{h} \in Y_{h}^{l}$, see e.g., \cite[Theorem 3.2.6]{Cia78},
\begin{equation}
\label{inv} \| \bv_{h} \|_{W^{m,p}(K)} \leq c_{\mathrm{inv}}
h_K^{n-m-d\left(\frac{1}{q}-\frac{1}{p}\right)}
\|\bv_{h}\|_{W^{n,q}(K)},
\end{equation}
where $0\leq n \leq m \leq 1$, $1\leq q \leq p \leq \infty$, and $h_K$
is the diameter of~$K \in \mathcal T_h$.
Let ${\boldsymbol V}=\left\{ {\boldsymbol \chi} \in H_0^1(\Omega)^d \mid
\nabla\cdot{\boldsymbol\chi} =0\right\}$. We consider a modified Stokes projection that was introduced in \cite{grad-div1} and that we denote by $\bs_h^m:{\boldsymbol V}\rightarrow {\boldsymbol V}_{h,l}$
satisfying
\begin{eqnarray}\label{stokespro_mod_def}
(\nabla \bs_h^m,\nabla \bvar_h)=(\nabla \bu,\nabla \bvar_h),\quad \forall \, \,
\bvar_{h} \in \bV_{h,l},
\end{eqnarray}
and the following error bound, see \cite{grad-div1}:
\begin{equation}
\|\bu-\bs_h^m\|_0+h\|\bu-\bs_h^m\|_1\le C\|\bu\|_j h^j,\qquad
1\le j\le l+1.
\label{stokespro_mod}
\end{equation}
From \cite{chenSiam}, we also have
\begin{align}
\|\nabla \bs_h^m\|_{L^\infty}\le C\|\nabla \bu\|_{L^\infty} \label{cotainfty1},
\end{align}
where $C$ does not depend on $\nu$ and~\cite[Lemma~3.8]{bosco_titi_yo}
\begin{align}
\label{cota_sh_inf_mu}
\|\bs_h^m\|_{L^\infty}  & \le C(\|\bu\|_{d-2}\|\bu\|_2)^{1/2},
\\
\label{la_cota_mu}
\|\nabla\bs_h^m\|_{L^{2d/(d-1)}} & \le
 C\bigl(\|\bu\|_1\|\bu\|_2\bigr)^{1/2},
\end{align}
where the constant~$C$ is independent of~$\nu$.


As mentioned in the introduction our aim is to get error bounds with constants independent on inverse powers of $\nu$. To this end, instead of starting with the snapshots of
a Galerkin method, we start from the snapshots of a stabilized method. We consider the method analyzed in \cite{nos_lps}:
the approximation of the solution of (\ref{NS}) with the implicit Euler method in time and a LPS method with LPS stabilization of the gradient of the velocity and the pressure. 

Assuming $l \geq 2$, the method reads: given  {${\bu}_h^0$ an approximation to $\bu^0$ in ${\boldsymbol X}_h^l$}, find $(\bu_h^{n+1},p_h^{n+1})\in {\boldsymbol X}_h^l\times Q_h^l$ such that {for $n\ge 0$}
\begin{eqnarray}\label{eq:gal_est}
\lefteqn{\hspace*{-16em}\left(\frac{\bu_h^{n+1}-\bu_h^n}{\Delta t},\bv_h\right)+\nu(\nabla \bu_h^{n+1},\nabla \bv_h)+b(\bu_h^{n+1},\bu_h^{n+1},\bv_h)-(p_h^{n+1},\nabla \cdot \bv_h)}
\nonumber\\
+
S_h(\bu_h^{n+1},\bv_h)&=&({\boldsymbol f}^{n+1},\bv_h) \quad \forall \bv_h\in {\boldsymbol X}_h^l,\\
(\nabla \cdot \bu_h^{n+1},q_h)+s_{\rm pres}(p_h^{n+1},q_h)&=&0 \quad \forall q_h\in Q_h^l.\nonumber
\end{eqnarray}
In \eqref{eq:gal_est} $(\bu_h^n,p_h^{n})$ is the stabilized approximation at time $t_n=n\Delta t$,  $\Delta t$
is the time step and $b_h(\cdot,\cdot,\cdot)$  is defined in the following way
$$
b_{h}(\bu_{h},\bv_{h},\bw_{h}) =((\bu_{h}\cdot \nabla ) \bv_{h}, \bw_{h})+ \frac{1}{2}( \nabla \cdot (\bu_{h})\bv_{h},\bw_{h}),
\quad \, \forall \, \bu_{h}, \bv_{h}, \bw_{h} \in {\boldsymbol X}_h^l.
$$
It is straightforward to verify that $b_h$ enjoys the skew-symmetry property
\begin{equation}\label{skew}
b_h(\bu,\bv,\bw)=-b_h(\bu,\bw,\bv) \qquad \forall \, \bu, \bv, \bw\in H_0^1(\Omega)^d.
\end{equation}
Also,
\begin{eqnarray}
S_h(\bu_h^{n+1},\bv_h)&=& \sum_{K\in \mathcal{T}_h}
\tau_{\nu,K}\left(\sigma^*_h (\nabla \bu_h^{n+1}), \sigma_h^*(\nabla \bv_h)\right)_K,\nonumber\\
s_{\rm pres}(p_h^{n+1},q_h)&=&\sum_{K\in \mathcal T_h}\tau_{p,K}(\sigma^*_h(\nabla p_h^{n+1}),\sigma^*_h(\nabla q_h))_K,\label{eq:spres}
\end{eqnarray}
where $\tau_{\nu,K}$ and $\tau_{p,K}$ are the gradient of the velocity and pressure stabilization parameters, respectively.
In addition,   $\sigma^*_h=Id-\sigma_h^{l-1}$, where $\sigma_h^{j}$ is a locally stable projection or interpolation operator from $L^2(\Omega)^d$ on ${\boldsymbol Y}_h^{j}.$
It will be assumed that
\begin{equation}\label{eq:tau_p2}
\alpha_1h_K\le \tau_{p,K}\le \alpha_2 h_K,
\end{equation}
and
\begin{equation}\label{eq:mu_2}
c_1h_K\le \tau_{\nu,K}\le c_2 h_K,
\end{equation}
with nonnegative constants $\alpha_1, \alpha_2, c_1, c_2$.
The following notation will be used
\begin{equation}\label{eq:norm_tau_p}
(f,g)_{\tau}=\sum_{K\in {\mathcal T}_h}\tau_{K}(f,g)_K, \quad
\|f\|_{\tau} = (f,f)_{\tau}^{1/2},
\end{equation}
where $\tau$ denotes either $\tau_p$ or $\tau_{\nu}$. Following \cite{ahmed_etal,Cha_maca}, we consider an approximation $\hat\bu_h^{n}\in {\boldsymbol X}_h^l$ of $\bu^n$ satisfying
\begin{equation}\label{eq:approxIMAJNA}
(\bu^n-\hat\bu_h^{n},\bv_h)=0,\quad\forall\bv_h\in{\boldsymbol Y}_h^{l-1}.
\end{equation}

Assume $\nu\le h$ and
\begin{eqnarray*}\label{eq:cond_dif}
C \|\hat\bu_h^{n+1}\|_{L^\infty}^2\left( \max_{K\in\mathcal T_h} \tau_{p,K}\right)\left(\max_{K\in\mathcal T_h} \tau_{\nu,K}^{-1}\right)\le \frac {1}{16},
\end{eqnarray*}
and let the parameter $\varepsilon$  be chosen
sufficiently small so that
\begin{equation}
\label{eq:epsilon}
C \varepsilon h \|\hat\bu_h^{n+1}\|_{L^\infty}^2\left(\max_{K\in\mathcal T_h} \tau_{\nu,K}^{-1}\right)\le \frac {1}{16}.
\end{equation}
Let
$$
M_u=1+C\|\bu\|_{L^\infty(H^3)}\left(1+\|\bu\|_{L^\infty(H^3)}\right)
$$
and
\begin{eqnarray*}\label{eq:tilde_Kb}
K_{u,p}=\left(\left(1+\varepsilon^{-1}+\|\bu\|_{L^\infty(H^3)}^2\right)\|\bu\|_{L^\infty(H^{l+1})}^2+\|\partial_t\bu\|_{L^\infty(H^{l+1})}^2
+\|p\|_{L^\infty(H^{l+1})}^2\right),
\end{eqnarray*}
$\varepsilon$ being the value in~(\ref{eq:epsilon}). Then, the following bound is proved in \cite[Theorem 6.1]{nos_lps}: for $T=M\Delta t$, $n\le M$
 \begin{eqnarray}\label{eq:err_after_gron_22b}
&&\|\bu^n-\bu_h^{n}\|_0^2+h^2\|\nabla(\bu^n-\bu_h^{n})\|_0^2
+\Delta t\sum_{j=1}^n\|\sigma_h^*(\nabla  (\bu^{j}-\bu_h^j))\|_{\tau_\nu}^2
\qquad
  \\
&&\quad+\Delta t\sum_{j=1}^n\|\sigma_h^*(\nabla  (p^j-p_h^{j}))\|_{\tau_p}^2\le C e^{2T M_u}\left(  T  K_{u,p} h^{2l+1}+ (\Delta t)^2\int_{t_0}^{t_n}\|\partial_{tt}\bu\|_{0}^2
~dt\right).\nonumber
\end{eqnarray}
From \eqref{eq:err_after_gron_22b} we can write
\begin{eqnarray}\label{eq:cota_est}
&&\|\bu^n-\bu_h^n\|_0+ h\|\bu^n-\bu_h^n\|_1+\left(\Delta t\sum_{j=1}^n S_{h}( \bu^j-\bu_h^{j},\bu^j-\bu_h^{j})\right)^{1/2}
\nonumber\\
&&\qquad+\left(\Delta t\sum_{j=1}^n s_{\rm pres}( p^j-p_h^{j},p^j-p_h^{j})\right)^{1/2}\nonumber\\
&&\ \quad\le C(\bu,p,l+1) (h^{l+1/2}+\Delta t), \quad 1\le n\le M,
\end{eqnarray}
where the constant $C(\bu,p,l+1)$ does not depend on inverse powers of $\nu$. From \cite[Remark 6.2]{nos_lps} and \cite{sim_pre_nos}
the following bound for the $L^2$ error of the pressure can be obtained for $1\le n\le M$
\begin{eqnarray}\label{eq:cota_est_pre}
\|p^{n}-p_h^{n}\|_0+ h\|p^{n}-p_h^{n}\|_1\le C(\bu,p,l+1) h^{-1/2}(h^{l+1/2}+\Delta t).
\end{eqnarray}
\section{Proper orthogonal decomposition}\label{sec:POD}
We will consider a proper orthogonal decomposition (POD) method.
Let us fix $T>0$ and $M>0$ and take $\Delta t=T/M$ and let us consider the following spaces
$$
{\cal \bV}=<\bu_h^1,\ldots,\bu_h^M>,\quad {\cal W}=<p_h^{1},\ldots,p_h^{M}>.
$$
Let $d_v$ be the dimension of the space $\cal \bV$ and let $d_p$ be the dimension of the space ${\cal W}$.

Let $K_v$, $K_p$ be the correlation matrices corresponding to the snapshots $K_v=((k_{i,j}^v))\in {\mathbb R}^{M\times M}$,
$K_p=((k_{i,j}^p))\in {\mathbb R}^{M\times M}$ where
$$
k_{i,j}^v=\frac{1}{M}(\bu_h^i,\bu_h^j),\quad k_{i,j}^p=\frac{1}{M}(p_h^{i},p_h^{j}),
$$
and $(\cdot,\cdot)$ is the inner product in $L^2(\Omega)^d$. Following \cite{kunisch} we denote by
$ \lambda_1\ge  \lambda_2,\ldots\ge \lambda_{d_v}>0$ the positive eigenvalues of $K_v$ and by
$\bv_1,\ldots,\bv_{d_v}\in {\mathbb R}^{M}$ the associated eigenvectors. Analogously, we denote by
$ \gamma_1\ge  \gamma_2,\ldots\ge \gamma_{d_p}>0$ the positive eigenvalues of $K_p$ and by
$\bw_1,\ldots,\bw_{d_p}\in {\mathbb R}^{M}$ the associated eigenvectors. Then, the (orthonormal) POD bases are given by
\begin{eqnarray}\label{lachi}
\bvar_k=\frac{1}{\sqrt{M}}\frac{1}{\sqrt{\lambda_k}}\sum_{j=1}^M v_k^j \bu_h(\cdot,t_j),\quad
\psi_k=\frac{1}{\sqrt{M}}\frac{1}{\sqrt{\gamma_k}}\sum_{j=1}^M w_k^j p_h(\cdot,t_{j}),
\end{eqnarray}
where $v_k^j$ is the $j$-th component of the eigenvector $\bv_k$ (respectively $w_k^j$ is the $j$-th component of the eigenvector $\bw_k$) and the following error formulas hold, see \cite[Proposition~1]{kunisch}
\begin{eqnarray}\label{eq:cota_pod_0}
\frac{1}{M}\sum_{j=1}^M\left\|\bu_h^j-\sum_{k=1}^r(\bu_h^j,\bvar_k)\bvar_k\right\|_{0}^2&=&\sum_{k=r+1}^{d_v}\lambda_k,\\
\frac{1}{M}\sum_{j=1}^M\left\|p_h^{j}-\sum_{k=1}^r(p_h^{j},\psi_k)\psi_k\right\|_{0}^2&=&\sum_{k=r+1}^{d_p}\gamma_k,\label{eq:cota_pod_0_pre}
\end{eqnarray}
where we have used the notation $\bu_h^j=\bu_h(\cdot,t_j)$, $p_h^{j}=p_h(\cdot,t_{j})$.

Denoting by $S$ the stiffness matrix for the POD basis: $S^v=((s_{i,j}^v))\in {\mathbb R}^{d_v\times d_v}$, $S^p=((s_{i,j}^p))\in {\mathbb R}^{d_p\times d_p}$  with
$s_{i,j}^v=(\nabla \bvar_i,\nabla \bvar_j)$
and $s_{i,j}^p=(\nabla \psi_i,\nabla \psi_j)$,
then for any $\bv \in {\cal \bV}$, $w\in {\cal W}$ the following inverse inequalities hold, see \cite[Lemma 2]{kunisch}
\begin{equation}\label{eq:inv_S}
||\nabla \bv ||_0\le \sqrt{\|S^v\|_2}\|\bv\|_0,\quad
||\nabla w ||_0\le \sqrt{\|S^p\|_2}\|w\|_0,
\end{equation}
where $\|\cdot\|_2$ denotes the spectral norm of the matrix.

From \eqref{eq:inv_S}, applying \eqref{eq:cota_pod_0} we get
\begin{eqnarray}\label{inv_1}
&&\frac{1}{M}\sum_{j=1}^M\left\|\nabla \bu_h^j-\sum_{k=1}^r(\bu_h^j,\bvar_k)\nabla \bvar_k\right\|_0^2\nonumber\\
&&\quad\le \frac{\|S^v\|_2}{M}\sum_{j=1}^M\left\|\bu_h^j-\sum_{k=1}^r(\bu_h^j,\bvar_k)\bvar_k\right\|_0^2\le \|S^v\|_2\sum_{k=r+1}^{d_v} \lambda_k.
\end{eqnarray}
Instead of \eqref{inv_1} we can also apply the following result that is taken from \cite[Lemma 3.2]{iliescu}
\begin{eqnarray}\label{inv_2}
\frac{1}{M}\sum_{j=1}^M\left\|\nabla\bu_h^j-\sum_{k=1}^r(\bu_h^j,\bvar_k)\nabla\bvar_k\right\|_0^2=
\sum_{k=r+1}^{d_v}\lambda_k \|\nabla\bvar_k\|_0^2.
\end{eqnarray}
Analogously, we obtain
\begin{eqnarray}\label{inv_1_p}
\frac{1}{M}\sum_{j=1}^M\left\|\nabla p_h^{j}-\sum_{k=1}^r(p_h^{j},\psi_k)\nabla\psi_k\right\|_0^2\leq
\|S^p\|_2\sum_{k=r+1}^{d_p}\gamma_k,
\end{eqnarray}
and
\begin{eqnarray}\label{inv_2_p}
\frac{1}{M}\sum_{j=1}^M\left\|\nabla p_h^{j}-\sum_{k=1}^r(p_h^{j},\psi_k)\nabla\psi_k\right\|_0^2=
\sum_{k=r+1}^{d_p}\gamma_k \|\nabla\psi_k\|_0^2.
\end{eqnarray}
In the sequel we will denote by
$$
{\cal \bV}^r=<\bvar_1,\bvar_2,\ldots,\bvar_r>,\quad{\cal W}^r=<\psi_1,\psi_2,\ldots,\psi_r>,
$$
and by $P_r^v$, $P_r^p$ the $L^2$-orthogonal projection onto ${\cal \bV}^r$ and ${\cal W}^r$, respectively.

\subsection{A priori bounds for the orthogonal projection onto ${\cal \bV}^r$}
In this section we will prove some a priori bounds for the orthogonal projection $P_r^v \bu_h^j$, $j=0,\cdots,M$ that are
obtained from a priori bounds for the Galerkin approximation $\bu_h^j$, $j=0,\cdots,M$. Then, we start getting a priori
bounds for the stabilized approximation $\bu_h^n$. We follow the same arguments we introduced in \cite{pod_da_nos}.
We start with the $L^\infty$ norm, using \eqref{inv}, \eqref{cota_sh_inf_mu}, \eqref{eq:cota_est} and  \eqref{stokespro_mod}
we get
\begin{eqnarray}\label{eq:uh_infty}
\|\bu_h^j\|_{L^\infty}&\le& \|\bu_h^j-\bs_h^m(\cdot,t_j)\|_{L^\infty}+\|\bs_h^m(\cdot,t_j)\|_{L^\infty}
\nonumber\\
&\le& C h^{-d/2}\|\bu_h^j-\bs_h^m(\cdot,t_j)\|_0+C(\|\bu^j\|_{d-2}\|\bu^j\|_2)^{1/2}\nonumber\\
&\le& C h^{-d/2}C(\bu,p,2)(h^{3/2}+\Delta t)+C(\|\bu^j\|_{d-2}\|\bu^j\|_2)^{1/2}\nonumber\\
&\le& C_{\bu,{\rm inf}}:=C\left( C(\bu,p,2)+(\|\bu^j\|_{d-2}\|\bu^j\|_2)^{1/2}\right),
\end{eqnarray}
whenever we assume the following condition holds for the time step
\begin{eqnarray}\label{eq:Delta_t0}
\Delta t \le C h^{d/2}.
\end{eqnarray}

Now we bound the $L^\infty$ norm of the gradient, using \eqref{inv}, \eqref{cotainfty1}, \eqref{eq:cota_est} and  \eqref{stokespro_mod}
we get
\begin{eqnarray}\label{eq:grad_uh_infty}
\|\nabla\bu_h^j\|_{L^\infty}&\le& \|\nabla\bu_h^j-\nabla\bs_h^m(\cdot,t_j)\|_{L^\infty}+\|\nabla\bs_h^m(\cdot,t_j)\|_{L^\infty}
\nonumber\\
&\le& C h^{-d/2}\|\bu_h^j-\bs_h^m(\cdot,t_j)\|_1+C\|\nabla \bu^j\|_{L^\infty}\nonumber\\
&\le& C h^{-d/2}C(\bu,p,3)h^{-1}\left(h^{5/2}+\Delta t \right)+C\|\nabla \bu^j\|_{L^\infty}\nonumber\\
&\le& C_{\bu,1,{\rm inf}}:= C\left(C(\bu,p,3)+\|\nabla  \bu\|_{L^\infty(L^\infty)}\right),
\end{eqnarray}
whenever we assume the following condition holds for the time step
\begin{eqnarray}\label{eq:Delta_t}
\Delta t \le C h^{(d+2)/2}.
\end{eqnarray}
Since condition \eqref{eq:Delta_t} implies \eqref{eq:Delta_t0} in the sequel we will assume the stronger condition \eqref{eq:Delta_t} holds.

Finally, we bound the $L^{2d/(d-1)}$ norm.  Using \eqref{inv}, \eqref{la_cota_mu}, \eqref{eq:cota_est} and  \eqref{stokespro_mod}
and assuming again condition \eqref{eq:Delta_t} holds (indeed the weaker condition $\Delta t \le C h^{3/2}$ would be enough) we get
\begin{eqnarray}\label{eq:uh_2d_dmenosuno}
\|\nabla \bu_h^j\|_{L^{2d/(d-1)}}&\le& \|\nabla(\bu_h^j-\bs_h^m(\cdot,t_j))\|_{L^{2d/(d-1)}}+\|\nabla \bs_h^m(\cdot,t_j)\|_{L^{2d/(d-1)}}
\nonumber\\
&\le& C h^{-1/2}\|\bu_h^j-\bs_h^m(\cdot,t_j)\|_1+C\bigl(\|\bu\|_1\|\bu\|_2\bigr)^{1/2}\nonumber\\
&\le& C h^{-1/2}C(\bu,p,2)h^{-1}(h^{3/2}+\Delta t)+C\bigl(\|\bu\|_1\|\bu\|_2\bigr)^{1/2}\nonumber\\
&\le& C_{\bu,{\rm ld}}:=C\left( C(\bu,p,2)+\bigl(\|\bu\|_1\|\bu\|_2\bigr)^{1/2}\right).
\end{eqnarray}
Now, we prove a priori bounds in the same norms for $P_r^v \bu_h^j$. We first observe that
$$
P_r^v \bu_h^j=(P_r^v \bu_h^j-\bu_h^j)+\bu_h^j.
$$
Since we have already proved error bounds for the second term on the right-hand side above we only need to bound the first one. To this end
we observe that it is easy to get
\begin{equation}\label{estasi}
\bu_h^j-P_r^v \bu_h^j=\sqrt{M}\sum_{k=r+1}^{d_v} \sqrt{\lambda_k} v_k^j\bvar_k.
\end{equation}
And then
\begin{eqnarray}\label{cota_bos1}
\|\bu_h^j-P_r^v \bu_h^j\|_0&=&\sqrt{M}\left(\sum_{k=r+1}^{d_v}\lambda_k|v_k^j|^2\right)^{1/2}\le \sqrt{M}\sqrt{\lambda_{r+1}}
\left(\sum_{k=r+1}^{d_v}|v_k^j|^2\right)^{1/2}\nonumber\\
&\le& \sqrt{M}\sqrt{\lambda_{r+1}},
\end{eqnarray}
where in the last inequality we have used that
\begin{eqnarray}\label{menor1}
\left(\sum_{k=r+1}^{d_v}|v_k^j|^2\right)^{1/2}\le 1
\end{eqnarray}
 since the matrix with columns
the vectors $\bv_k$ can be enlarged to an $M\times M$ orthogonal matrix.
Applying inverse inequality \eqref{inv}, \eqref{eq:uh_infty} and \eqref{cota_bos1} we get
\begin{eqnarray}\label{eq:cotaPlinf2}
\|P_r^v \bu_h^j\|_{L^\infty}&\le& \|\bu_h^j\|_{L^\infty}+C h^{-d/2}\|\bu_h^j-P_r^v \bu_h^j\|_0\nonumber\\
&\le& C_{\rm inf}:=
C_{\bu,{\rm inf}}+C h^{-d/2}\sqrt{M}\sqrt{\lambda_{r+1}}.
\end{eqnarray}
Now, we observe that from \eqref{inv_1} we get
$$
\|\nabla(\bu_h^j-P_r^v \bu_h^j)\|_0\le \sqrt{M}\|S^v\|_2^{1/2}\left(\sum_{k=r+1}^{d_v} \lambda_k\right)^{1/2}.
$$
Applying this inequality together with inverse inequality \eqref{inv} and \eqref{eq:grad_uh_infty} we obtain
\begin{eqnarray}\label{eq:cotanablaPlinf}
\|\nabla P_r^v \bu_h^j\|_{L^\infty}\le C_{1,\rm inf}:=C_{\bu,1,{\rm inf}}+C h^{-d/2}\sqrt{M}\|S^v\|_2^{1/2}\left(\sum_{k=r+1}^{d_v} \lambda_k\right)^{1/2}.
\end{eqnarray}
 Finally, arguing in the same way but applying \eqref{eq:uh_2d_dmenosuno} instead of \eqref{eq:grad_uh_infty} we can write
 \begin{eqnarray}\label{eq:cotanablaPld}
\|\nabla P_r^v\bu^j\|_{L^{2d/(d-1)}}&\le& C_{\rm ld}:=C_{\bu,{\rm ld}}+Ch^{-1/2}\sqrt{M}\|S^v\|_2^{1/2}\left(\sum_{k=r+1}^{d_v} \lambda_k\right)^{1/2}.
\end{eqnarray}

\section{A POD stabilized method from a non inf-sup stable FOM}\label{sec:LPS-ROM}
For a given initial condition $\bu_r^0$  we consider the following POD stabilized method in which, for simplicity, as a time integrator we apply the implicit Euler method. For $n\ge 1$, find $(\bu_r^n,p_r^{n})\in {\cal \bV}^r\times{\cal W}^r$ such that
\begin{eqnarray}\label{eq:pod_method1}
&&\left(\frac{\bu_r^{n}-\bu_r^{n-1}}{\Delta t},\bvar\right)+\nu(\nabla \bu_r^n,\nabla\bvar)+b_h(\bu_r^n,\bu_r^n,\bvar)
-(p_r^{n},\nabla \cdot \bvar)
\nonumber\\
&&\quad\quad\quad\ \ +S_h(\bu_r^n,\bvar)+\mu(\nabla \cdot\bu_r^n,\nabla \cdot\bvar)=(\bff^{n},\bvar),\quad \forall \bvar\in {\cal \bV}^r,\\\nonumber
&&(\nabla \cdot \bu_r^n,\psi)+s_{\rm pres}(p_r^{n},\psi)=0,\quad \quad\quad\quad \ \forall \psi\in {\cal W}^r,
\end{eqnarray}
where the gradient of the velocity and pressure stabilization terms are defined in \eqref{eq:spres}, and $\mu$ is the positive grad-div stabilization parameter.

To get the error bounds of the method we will compare $\bu_r^n$ with $P_r^v \bu^n_h$ and $p_r^n$ with $P_r^p p^n_h$.
Let us denote by
$$
\bbeta_h^n=P_r^v\bu^n_h-\bu_h^{n},\quad \xi_h^n=P_r^p p^{n}_h-p_h^n.
$$
It is easy to get
\begin{eqnarray}\label{eq:prov_prop}
&&\left(\frac{P_r^v \bu^{n}_h-P_r^v \bu^{n-1}_h}{\Delta t},\bvar\right)+\nu(\nabla  P_r^v \bu^n_h,\nabla\bvar)+b_h(P_r^v \bu^n_h,P_r^v \bu^n_h,\bvar)
\nonumber\\
&&\quad -(P_r^p p^{n}_h,\nabla \cdot \bvar)+S_h(P_r^v \bu_h^n,\bvar)+\mu(\nabla \cdot P_r^v \bu^n_h,\nabla \cdot\bvar)=(\bff^{n},\bvar)
 \nonumber\\
&&\quad +\nu(\nabla\bbeta_h^n,\nabla\bvar)+S_h(\bbeta_h^n,\bvar)-(\xi_h^n,\nabla \cdot \bvar)+\mu(\nabla \cdot\bbeta_h^n,\nabla \cdot \bvar)\nonumber\\
&&\quad+b_h(P_r^v\bu^n_h,P_r^v \bu^n_h,\bvar)-b_h(\bu^{n}_h,\bu^{n}_h,\bvar),\quad \forall \bvar\in {\cal \bV}^r,\\\nonumber
&&\quad(\nabla \cdot P_r^v \bu^n,\psi)+s_{\rm pres}(P_r p^{n},\psi)=(\nabla \cdot \bbeta_h^n,\psi)+s_{\rm pres}(\xi_h^n,\psi),\quad  \forall \psi\in {\cal W}^r.
\end{eqnarray}
Let us denote by
$$
\be_r^n=\bu_r^n-P_r^v \bu^n_h,\quad z_r^n=p_r^n-P_r^p p^n_h.
$$
Subtracting \eqref{eq:prov_prop} from \eqref{eq:pod_method1} and taking $\bvar=\be_r^n$ and $\psi=z_r^{n}$ we get
\begin{eqnarray}\label{eq:error}
&&\frac{1}{2\Delta t}\left(\|\be_r^n\|_0^2-\|\be_r^{n-1}\|_0^2\right)+\nu\|\nabla \be_r^n\|_0^2+S_h(\be_r^n,\be_r^n)
+\mu\|\nabla \cdot \be_r^n\|_0^2+s_{\rm pres}(z_r^{n},z_r^{n})\nonumber\\
&&\quad\le \left(b_h(P_r^v  \bu^n_h,P_r^v \bu^n_h,\be_r^n)-b_h(\bu_r^n,\bu_r^n,\be_r^n)\right)-\nu(\nabla\bbeta_h^n,\nabla\be_r^n)
-S_h(\bbeta_h^n,\be_r^n)\nonumber\\
&&\quad+(\xi_h^n,\nabla \cdot \be_r^n)-\mu(\nabla \cdot\bbeta_h^n,\nabla \cdot \be_r^n)-\left(b_h(P_r^v\bu^n_h,P_r^v \bu^n_h,\be_r^n)-b_h(\bu^{n}_h,\bu^{n}_h,\be_r^n)\right)\nonumber\\
&&\quad -(\nabla \cdot \bbeta_h^n,z_r^n)-s_{\rm pres}(\xi_h^n,z_r^n)\nonumber\\
&&\quad\le I+II+III+IV+V+VI+VII+VIII.
\end{eqnarray}
We will bound the terms on the right-hand side of \eqref{eq:error}.
We first observe that using the skew-symmetric property \eqref{skew}, \eqref{eq:cotaPlinf2} and \eqref{eq:cotanablaPlinf} we get
\begin{eqnarray}\label{eq:cota_er_1}
|I|\le |b_h(\be_r^n,P_r^v \bu_h^n,\be_r^n)|&\le& \|\nabla P_r^v\bu_h^n\|_{L^\infty}\|\be_r^n\|_0^2+\frac{1}{2}\|\nabla \cdot \be_r^n\|_0
\|P_r^v\bu_h^n\|_{L^\infty}\|\be_r^n\|_0\nonumber\\
&\le& C_{1,{\rm inf}}\|\be_r^n\|_0^2+\frac{C_{\rm inf}}{2}\|\nabla \cdot \be_r^n\|_0\|\be_r^n\|_0\nonumber\\
&\le& \left(C_{1,{\rm inf}}+\frac{C_{\rm inf}^2}{4\mu}\right)\|\be_r^n\|_0^2+\frac{\mu}{4}\|\nabla \cdot \be_r^n\|_0^2.
\end{eqnarray}
For the second term we obtain
\begin{eqnarray}\label{eq:cota_er_2}
|II|\le \frac{\nu}{2}\|\nabla \be_r^n\|_0^2+\frac{\nu}{2}\|\nabla\bbeta_h^n\|_0^2.
\end{eqnarray}
For the third
\begin{eqnarray}\label{eq:cota_er_3}
|III|\le \frac{1}{2}S_h(\be_r^n,\be_r^n)+\frac{1}{2}S_h(\bbeta_h^n,\bbeta_h^n).
\end{eqnarray}
For the forth
\begin{eqnarray}\label{eq:cota_er_4}
|IV|\le \frac{\mu}{8}\|\nabla \cdot \be_r^n\|_0^2+\frac{2}{\mu}\|\xi_h^n\|_0^2.
\end{eqnarray}
For the fifth term we get
\begin{eqnarray}\label{eq:cota_er_5}
|V|\le\frac{\mu}{8}\|\nabla \cdot \be_r^n\|_0^2+2\mu \|\nabla \cdot \bbeta_h^n\|_0^2.
\end{eqnarray}
To bound the sixth term we apply Sobolev embeddings \eqref{sob1}, \eqref{eq:cotaPlinf2},  \eqref{diver_vol}, \eqref{eq:cotanablaPld},
\eqref{eq:uh_infty} and \eqref{eq:uh_2d_dmenosuno}
\begin{eqnarray}\label{eq:cota_er_6}
|VI|&\le& |b_h(P_r^v\bu_h^n,\bbeta_h^n,\be_r^n)|+|b_h(\bbeta_h^n,\bu_h^n,\be_r^n)|
\nonumber\\&\le&
\|P_r^v\bu_h^n\|_{L^\infty}\|\nabla \bbeta_h^n\|_0\|\be_r^n\|_0+\frac{1}{2}\|\nabla \cdot P_r^v\bu_h^n\|_{L^{2d/(d-1)}}
\| \bbeta_h^n\|_{L^{2d}}\|\be_r^n\|_0\nonumber\\
&&\quad +\|\bbeta_h^n\|_{L^{2d}} \|\nabla  \bu_h^n\|_{L^{2d/(d-1)}}\|\be_r^n\|_0+\frac{1}{2}
\|\nabla \cdot\bbeta_h^n\|_0\|\bu_h^n\|_{L^\infty}\|\be_r^n\|_0\nonumber\\
&\le& C_{\rm inf}\|\nabla \bbeta_h^n\|_0\|\be_r^n\|_0+\frac{1}{2}C C_{\rm ld}\|\nabla \bbeta_h^n\|_0\|\be_r^n\|_0
\nonumber\\
&&\quad +C\|\nabla \bbeta_h^n\|_0C_{\bu,\rm ld}\|\be_r^n\|_0+\frac{1}{2}\|\nabla \bbeta_h^n\|_0C_{\bu,\rm inf}\|\be_r^n\|_0\nonumber\\
&\le& C\|\nabla \bbeta_h^n\|_0^2+\frac{1}{2}\|\be_r^n\|_0^2.
\end{eqnarray}
To bound the seventh term we will use the properties of the direct method used to compute the snapshots. From \eqref{estasi}, \eqref{lachi}
and \eqref{eq:pod_method1} we can write
\begin{eqnarray}\label{eq:cota_er_7}
|VII|&\le& \left|\sum_{k={r+1}}^{d_v}\sqrt{M}\sqrt{\lambda_k}v_k^j\left(\nabla \cdot \bvar_k,z_r^n\right)\right|
=\left|\sum_{k={r+1}}^{d_v}v_k^j\sum_{s=1}^Mv_k^s \left(\nabla \cdot \bu_h^s,z_r^n\right)\right|\nonumber\\
&\le&\left|\sum_{k={r+1}}^{d_v}v_k^j\sum_{s=1}^Mv_k^s s_{\rm pres}(p_h^s,z_r^n)\right|\nonumber\\
&\le&\left|
\sum_{k={r+1}}^{d_v}v_k^j\sum_{s=1}^Mv_k^s s_{\rm pres}(p_h^s,p_h^s)^{1/2}s_{\rm pres}(z_r^n,z_r^n)^{1/2}\right|\nonumber\\
&\le&\frac{1}{4}s_{\rm pres}(z_r^n,z_r^n)+\left(\sum_{k={r+1}}^{d_v}v_k^j\sum_{s=1}^Mv_k^s s_{\rm pres}(p_h^s,p_h^s)^{1/2}\right)^2\nonumber\\
&\le&\frac{1}{4}s_{\rm pres}(z_r^n,z_r^n)+\left(\sum_{k={r+1}}^{d_v}|v_k^j|^2\right)\sum_{k={r+1}}^{d_v}\left(\sum_{s=1}^Mv_k^s s_{\rm pres}(p_h^s,p_h^s)^{1/2}\right)^2\nonumber\\
&\le&\frac{1}{4}s_{\rm pres}(z_r^n,z_r^n)+\sum_{k={r+1}}^{d_v}\left(\sum_{s=1}^M|v_k^s|^2\right)\left(\sum_{s=1}^Ms_{\rm pres}(p_h^s,p_h^s)\right)
\nonumber\\
&\le&\frac{1}{4}s_{\rm pres}(z_r^n,z_r^n)+(d_v-r)\sum_{k=1}^M s_{\rm pres}(p_h^k,p_h^k),
\end{eqnarray}
where in the last inequality we have used \eqref{menor1}.

For the last term on the right-hand side of \eqref{eq:error} we get
\begin{eqnarray}\label{eq:cota_er_8}
|VIII|&\le&\frac{1}{4}s_{\rm pres}(z_r^n,z_r^n)+s_{\rm pres}(\xi_h^n,\xi_h^n).
\end{eqnarray}
Inserting \eqref{eq:cota_er_1}, \eqref{eq:cota_er_2}, \eqref{eq:cota_er_3}, \eqref{eq:cota_er_4}, \eqref{eq:cota_er_5}, \eqref{eq:cota_er_6}, \eqref{eq:cota_er_7} and \eqref{eq:cota_er_8} into \eqref{eq:error} we obtain
\begin{eqnarray*}\label{eq:error2}
&&\frac{1}{\Delta t}\left(\|\be_r^n\|_0^2-\|\be_r^{n-1}\|_0^2\right)+\nu\|\nabla \be_r^n\|_0^2+S_h(\be_r^n,\be_r^n)
+\mu\|\nabla \cdot \be_r^n\|_0^2+s_{\rm pres}(z_r^{n},z_r^{n})\nonumber\\
&&\le \left(2C_{1,\rm inf}+\frac{C_{\rm inf}^2}{2\mu}+1\right)\|\be_r^n\|_0^2
+C(\nu+\mu+1)\|\nabla \bbeta_h^n\|_0^2+S_h(\bbeta_h^n,\bbeta_h^n)
\nonumber\\
&&\quad+\frac{4}{\mu}\|\xi_h^n\|_0^2+2(d_v-r)(\Delta t)^{-1}\sum_{k=1}^M \Delta t s_{\rm pres}(p_h^k,p_h^k)+2s_{\rm pres}(\xi_h^n,\xi_h^n).
\end{eqnarray*}
Summing over times we reach
\begin{eqnarray}\label{eq:error2}
&&\|\be_r^n\|_0^2+\nu\sum_{j=1}^n\Delta t \|\nabla \be_r^j\|_0^2+\sum_{j=1}^n\Delta t S_h(\be_r^j,\be_r^j)
+\mu\sum_{j=1}^n\Delta t\|\nabla \cdot \be_r^j\|_0^2\nonumber\\
&&\quad+\sum_{j=1}^n\Delta ts_{\rm pres}(z_r^{j},z_r^{j})\nonumber\\
&&\le \|\be_r^0\|_0^2+\sum_{j=1}^n\Delta t \left(2C_{1,\rm inf}+\frac{C_{\rm inf}^2}{2\mu}+1\right)\|\be_r^j\|_0^2
+\btau_n,
\end{eqnarray}
where
\begin{eqnarray}\label{eq:tau}
\btau_n&=&\sum_{j=1}^n\Delta tS_h(\bbeta_h^j,\bbeta_h^j)+C(\nu+\mu+1)\sum_{j=1}^n\Delta t\|\nabla \bbeta_h^j\|_0^2\nonumber\\
&&\quad+\sum_{j=1}^n\Delta t\left(\frac{4}{\mu}\|\xi_h^j\|_0^2
+2s_{\rm pres}(\xi_h^j,\xi_h^j)\right)\nonumber\\
&&\quad+2(d_v-r)T(\Delta t)^{-1}\sum_{k=1}^M \Delta t s_{\rm pres}(p_h^k,p_h^k)\nonumber\\
&=&\btau_1+\btau_2+\btau_3+\btau_4.
\end{eqnarray}
Denoting by
$$
C_u=2C_{1,\rm inf}+\frac{C_{\rm inf}^2}{2\mu}+1,
$$
assuming $\Delta t C_u\le 1/2$ and applying Gronwall's Lemma \cite[Lemma 5.1]{Hey-RanIV} we get
\begin{eqnarray}\label{eq:error3}
&&\|\be_r^n\|_0^2+\nu\sum_{j=1}^n\Delta t \|\nabla \be_r^j\|_0^2+\sum_{j=1}^n\Delta t S_h(\be_r^j,\be_r^j)+\mu\sum_{j=1}^n\Delta t\|\nabla \cdot \be_r^j\|_0^2\nonumber\\
&&\quad+\sum_{j=1}^n\Delta ts_{\rm pres}(z_r^{j},z_r^{j})\le e^{2T C_u}\left(\|\be_r^0\|_0^2+\btau_n\right).
\end{eqnarray}
To conclude we will get an error bound for $\btau_n$. For the first term in \eqref{eq:tau}, applying the definition of $S_h$, the $L^2$-stability of $\sigma_h^*$, \eqref{eq:mu_2} and \eqref{inv_1} we get
\begin{eqnarray}\label{eq:trun1}
\btau_1\le C h \sum_{j=1}^n\Delta t \|\nabla\bbeta_h^j\|_0^2\le C h\|S^v\|_2\sum_{k=r+1}^{d_v} \lambda_k.
\end{eqnarray}
For the second term in \eqref{eq:tau}, applying again \eqref{inv_1} we obtain
\begin{eqnarray}\label{eq:trun2}
\btau_2\le  C(\nu+\mu+1)\|S^v\|_2\sum_{k=r+1}^{d_v} \lambda_k.
\end{eqnarray}
For the third term, we again apply the $L^2$-stability of $\sigma_h^*$, \eqref{eq:tau_p2}, \eqref{eq:cota_pod_0_pre} and \eqref{inv_1_p}  to prove
\begin{eqnarray}\label{eq:trun3}
\btau_3\le C (\mu^{-1}+h\|S^p\|_2)  \sum_{k=r+1}^{d_p} \gamma_k.
\end{eqnarray}
Finally, for the last term, applying \eqref{eq:cota_est} we get
\begin{eqnarray}\label{eq:trun4}
\btau_4\le C(\bu,p,l+1)^2(d_v-r)(\Delta t)^{-1} (h^{2l+1}+(\Delta t)^2).
\end{eqnarray}
Inserting \eqref{eq:trun1}, \eqref{eq:trun2}, \eqref{eq:trun3} and \eqref{eq:trun4} into \eqref{eq:error3} we reach
\begin{eqnarray}\label{eq:error4}
&&\|\be_r^n\|_0^2+\nu\sum_{j=1}^n\Delta t \|\nabla \be_r^j\|_0^2+\sum_{j=1}^n\Delta t S_h(\be_r^j,\be_r^j)+\mu\sum_{j=1}^n\Delta t\|\nabla \cdot \be_r^j\|_0^2\\
&&\quad+\sum_{j=1}^n\Delta ts_{\rm pres}(z_r^{j},z_r^{j})\le e^{2T C_u}\left(\|\be_r^0\|_0^2+
C(\nu+\mu+1+h)\|S^v\|_2\sum_{k=r+1}^{d_v} \lambda_k\right.\nonumber\\
&&\quad \left. +C (\mu^{-1}+h\|S^p\|_2)  \sum_{k=r+1}^{d_p} \gamma_k+C(\bu,p,l+1)^2(d_v-r)(\Delta t)^{-1} (h^{2l+1}+(\Delta t)^2)  \right)\nonumber.
\end{eqnarray}
\begin{remark}
Let us observe that, contrary to other error bounds in the literature, apart from the last term on the right-hand side of  \eqref{eq:error4},
all the error bound is written in terms of the eigenvalues. The last term appears due to the fact that  non inf-sup stable elements
have been used both in the formulation of the direct and the POD methods. The reason for being able to write most of bound in terms of the eigenvalues is that we have compared the POD velocity approximation with $P_r^v \bu_h^n$ instead of $P_r^v \bu^n$, as it is standard. On the other hand, it is interesting to observe that the constants in the error bound \eqref{eq:error4} are independent on inverse powers of the viscosity.
\end{remark}
\begin{Theorem}
Let $\bu$ be the velocity in the Navier-Stokes equations \eqref{NS}, let $\bu_r$ be the LPS POD
stabilized approximation defined in \eqref{eq:pod_method1}, and assume that the solution $(\bu,p)$ of \eqref{NS} is regular enough. Then, the
following bound holds
\begin{eqnarray}\label{eq:cota_finalLPS}
\sum_{j=1}^n\Delta t \|\bu_r^j-\bu^j\|_0^2&\le& 3e^{2T C_u}\left(\|\be_r^0\|_0^2+
C(\nu+\mu+1+h)\|S^v\|_2\sum_{k=r+1}^{d_v} \lambda_k\right.\nonumber\\
&&\quad \left. +C (\mu^{-1}+h\|S^p\|_2)  \sum_{k=r+1}^{d_p} \gamma_k\right.\\
&&\quad\left.+C(\bu,p,l+1)^2(d_v-r)(\Delta t)^{-1} (h^{2l+1}+(\Delta t)^2)  \right)\nonumber\\
&&\quad +3TC(\bu,p,l+1)^2(h^{2l+1}+(\Delta t)^2)+3\sum_{k=r+1}^{d_v}\lambda_k.\nonumber
\end{eqnarray}
\end{Theorem}
\begin{proof}
Since $\sum_{j=1}^n\Delta t \|\be_r^j\|_0^2\le T \max_{1\le j\le n}\|\be_r^j\|_0^2$ and
\begin{eqnarray}\label{eq:lauso}
\sum_{j=1}^n\Delta t \|\bu_r^j-\bu^j\|_0^2&\le&
3\left(\sum_{j=1}^n\Delta t \|\be_r^j\|_0^2+\sum_{j=1}^n\Delta t\|P_r^v \bu_h^j-\bu_h^j\|_0^2\right.\nonumber\\
&&\quad\left.+\sum_{j=1}^n\Delta t\| \bu_h^j-\bu^j\|_0^2\right),
\end{eqnarray}
applying \eqref{eq:error4}, \eqref{eq:cota_pod_0} and \eqref{eq:cota_est} we finally obtain \eqref{eq:cota_finalLPS}.
\end{proof}
\begin{remark}
We observe that the error bound \eqref{eq:cota_finalLPS} has the component $$(\Delta t)^{-1} (h^{2l+1}+(\Delta t)^2)$$ that comes from the treatment of
the divergence term \eqref{eq:cota_er_7} when using non inf-sup stable elements. However, apart from this term, the $L^2$ error behaves in terms of the mesh size $h$ as the $L^2$ error of the direct method, i.e. as $h^{2l+1}$. This is not the case when one compares with $P_r^v \bu^n$. In that case, the $L^2$ error of the method is bounded by the $H^1$ norm of the error $P_r^v\bu^n-\bu^n$ and, consequently, can be bounded in terms of  the $H^1$  (instead of $L^2$) norm
of the direct method.
\end{remark}
\begin{remark}\label{rm:PresErrEstLPS}
For the method proposed in this section, denoting by
$$
Z_r^n=\sum_{j=1}^n \Delta t z_r^j,
$$
one can argue as in \cite{samuele_pod} to bound the pressure error in the norm defined in \cite{samuele_pod}:
$$
|||Z_r^n|||:=\sup_{\bvar_r\in {\cal \bV}^r}\frac{(Z_r^n,\nabla \cdot \bvar_r)}{\|\nabla \bvar_r\|_0}+s_{\rm pres}(Z_r^n,Z_r^n)^{1/2}.
$$
For the mixed finite element direct method \eqref{eq:gal_est} one has the discrete inf-sup condition, see \cite[Lemma 4.2]{ahmed_etal}:
$$
\forall q_h \in Q_h^l,\quad \|q_h\|_0\le \beta_0\left(\sup_{\bv_h\in {\boldsymbol X}_h^l}\frac{(q_h,\nabla \cdot \bv_h)}{\|\nabla \bv_h\|_0}
+s_{\rm pres}(q_h,q_h)^{1/2}\right).
$$
However, the above inequality has not been proved if one changes in the supremum ${\boldsymbol X}_h^l$ by ${\cal \bV}^r$, so that
in principle the norm $|||\cdot|||$ could be weaker than the $L^2$ norm.
\end{remark}
\section{A grad-div stabilized POD method from an inf-sup stable FOM}\label{sec:grad-div-ROM}
In this section we consider the case in which we start with snapshots of a direct method based on inf-sup stable elements. Since the snapshots
satisfy a discrete divergence-free condition in this case one can formulate the POD method with only velocity approximations. Following \cite{schneier} we will prove that a pressure approximation can be computed also in this case and we will prove error bounds for both the velocity and the pressure with constants independent on inverse powers of the viscosity.

Let $l \geq 2$, we consider the  MFE pair known as Hood--Taylor elements \cite{BF,hood0} $({\boldsymbol X}_h^l, Q_{h}^{l-1})$.
We recall the definition of the divergence-free space ${\boldsymbol V}_{h,l}$ \eqref{eq:V}.

For these elements a uniform inf-sup condition is satisfied (see \cite{BF}), that is, there exists a constant $\beta_{\rm is}>0$ independent of the mesh size $h$ such that
\begin{equation}\label{lbbh}
 \inf_{q_{h}\in Q_{h}^{l-1}}\sup_{\bv_{h}\in{\boldsymbol X}_h^l}
\frac{(q_{h},\nabla \cdot \bv_{h})}{\|\bv_{h}\|_{1}
\|q_{h}\|_{L^2/{\mathbb R}}} \geq \beta_{\rm{is}}.
\end{equation}
As a direct method we consider a Galerkin method with grad-div stabilization and the implicit Euler method in time.
Given  {${\bu}_h^0$ an approximation to $\bu^0$ in ${\boldsymbol X}_h^l$}, find $(\bu_h^{n+1},p_h^{n+1})\in {\boldsymbol X}_h^l\times Q_h^{l-1}$ such that {for $n\ge 0$}
\begin{eqnarray}\label{eq:gal_grad_div}
&&\left(\frac{\bu_h^{n+1}-\bu_h^n}{\Delta t},\bv_h\right)+\nu(\nabla \bu_h^{n+1},\nabla \bv_h)+b(\bu_h^{n+1},\bu_h^{n+1},\bv_h)-(p_h^{n+1},\nabla \cdot \bv_h)
\nonumber\\
&&\quad+
\mu(\nabla \cdot\bu_h^{n+1},\nabla \cdot \bv_h)=({\boldsymbol f}^{n+1},\bv_h) \quad \forall \bv_h\in {\boldsymbol X}_h^l,\\
&&(\nabla \cdot \bu_h^{n+1},q_h)=0 \quad \forall q_h\in Q_h^{l-1},\nonumber
\end{eqnarray}
where $\mu$ is the positive grad-div stabilization parameter.

It is well-known that considering the discrete divergence-free space $\bV_{h,l}$ we can remove the pressure from \eqref{eq:gal_grad_div} since
$\bu_h^{n+1}\in \bV_{h,l}$ satisfies
\begin{eqnarray}\label{eq:gal_grad_div2}
&&\left(\frac{\bu_h^{n+1}-\bu_h^n}{\Delta t},\bv_h\right)+\nu(\nabla \bu_h^{n+1},\nabla \bv_h)+b(\bu_h^{n+1},\bu_h^{n+1},\bv_h)
\nonumber\\
&&\quad+
\mu(\nabla \cdot\bu_h^{n+1},\nabla \cdot \bv_h)=({\boldsymbol f}^{n+1},\bv_h), \quad \forall \bv_h\in {\bV}_{h,l}.
\end{eqnarray}
For this method the following bounds hold, see \cite{NS_grad_div}
\begin{eqnarray}\label{eq:cota_grad_div}
\|\bu^n-\bu_h^n\|_0+ h\|\bu^n-\bu_h^n\|_1\le C(\bu,p,l+1) (h^{l}+\Delta t), \quad 1\le n\le M,
\end{eqnarray}
and
\begin{eqnarray}\label{eq:cota_grad_div_pre}
\left(\sum_{j=1}^n\Delta t \|p^j-p_h^j\|_0^2\right)^{1/2}\le C(\bu,p,l+1) (h^{l}+\Delta t), \quad 1\le n\le M,
\end{eqnarray}
where the constant $C(\bu,p,l+1)$ does not depend on inverse powers of $\nu$.

Arguing as before, we can get a priori bounds for the $L^2$ orthogonal projection $P_r^v \bu_h^n$. However, since now the FOM has a rate
of convergence of order $l$ instead of $l+1/2$, as the method in the previous section, there are some differences. We will assume $d=2$. 
In case $d=3$ one can prove a priori bounds
but using cubics (or higher) elements. In the error bound \eqref{eq:uh_infty} we have to replace $h^{3/2}$ by $h$.
 In the error bound \eqref{eq:grad_uh_infty} we have to replace $h^{5/2}$ by $h^2$. Finally, in the bound \eqref{eq:uh_2d_dmenosuno} we have to change $h^{3/2}$ by $h^2$ and
 $C(\bu,p,2)$ by $C(\bu,p,3)$.

We now consider the  grad-div POD model. For $n\ge 1$, find $\bu_r^n\in {\cal \bV}^r$ such that
\begin{eqnarray}\label{eq:pod_method2}
&&\left(\frac{\bu_r^{n}-\bu_r^{n-1}}{\Delta t},\bvar\right)+\nu(\nabla \bu_r^n,\nabla\bvar)+b_h(\bu_r^n,\bu_r^n,\bvar)
+\mu(\nabla \cdot\bu_r^n,\nabla \cdot\bvar)\nonumber\\
&&\quad=(\bff^{n},\bvar),\quad \forall \bvar\in {\cal \bV}^r.
\end{eqnarray}
We observe that $\bu_r^n$ belongs to the discrete divergence-free space $\bV_{h,l}$.

It is easy to get
\begin{eqnarray}\label{eq:prov_prop2}
&&\left(\frac{P_r^v \bu^{n}_h-P_r^v \bu^{n-1}_h}{\Delta t},\bvar\right)+\nu(\nabla  P_r^v \bu^n_h,\nabla\bvar)+b_h(P_r^v \bu^n_h,P_r^v \bu^n_h,\bvar)
\nonumber\\
&&\quad+\mu(\nabla \cdot P_r^v \bu^n_h,\nabla \cdot\bvar)\nonumber\\
&&=(\bff^{n},\bvar)
 +\nu(\nabla\bbeta_h^n,\nabla\bvar)+\mu(\nabla \cdot\bbeta_h^n,\nabla \cdot \bvar)\\
&&\quad+b_h(P_r^v\bu^n_h,P_r^v \bu^n_h,\bvar)-b_h(\bu^{n}_h,\bu^{n}_h,\bvar),\quad \forall \bvar\in {\cal \bV}^r.\nonumber
\end{eqnarray}
Subtracting \eqref{eq:prov_prop2} from \eqref{eq:pod_method2} and denoting as before by $\be_r^n=\bu_r^n-P_r^v\bu_h^n$
we get
\begin{eqnarray}\label{eq:error_need}
&&\left(\frac{\be_r^n -\be_r^{n-1}}{\Delta t},\bvar\right)+\nu(\nabla  \be_r^n ,\nabla\bvar)+b_h(\be_r^n,\be_r^n,\bvar)
+\mu(\nabla \cdot \be_r^n,\nabla \cdot\bvar)\nonumber\\
&&=
 -\nu(\nabla\bbeta_h^n,\nabla\bvar)-\mu(\nabla \cdot\bbeta_h^n,\nabla \cdot \bvar)-b_h(P_r^v\bu^n_h,P_r^v \bu^n_h,\bvar)\nonumber\\
 &&\quad+b_h(\bu^{n}_h,\bu^{n}_h,\bvar),\quad \forall \bvar\in {\cal \bV}^r.
\end{eqnarray}
Taking $\bvar=\be_r^n$ it is easy to obtain
\begin{eqnarray}\label{eq:error2}
&&\frac{1}{2\Delta t}\left(\|\be_r^n\|_0^2-\|\be_r^{n-1}\|_0^2\right)+\nu\|\nabla \be_r^n\|_0^2)
+\mu\|\nabla \cdot \be_r^n\|_0^2\nonumber\\
&&\quad\le -\left(b_h(P_r^v  \bu^n_h,P_r^v \bu^n_h,\be_r^n)-b_h(\bu_r^n,\bu_r^n,\be_r^n)\right)-\nu(\nabla\bbeta_h^n,\nabla\be_r^n)
\nonumber\\
&&\quad-\mu(\nabla \cdot\bbeta_h^n,\nabla \cdot \be_r^n)-\left(b_h(P_r^v\bu^n_h,P_r^v \bu^n_h,\be_r^n)-b_h(\bu^{n}_h,\bu^{n}_h,\be_r^n)\right).
\end{eqnarray}
Arguing exactly as before we get
\begin{eqnarray}\label{eq:error2}
&&\|\be_r^n\|_0^2+\nu\sum_{j=1}^n\Delta t \|\nabla \be_r^j\|_0^2
+\mu\sum_{j=1}^n\Delta t\|\nabla \cdot \be_r^j\|_0^2\\
&&\le \|\be_r^0\|_0^2+\sum_{j=1}^n\Delta t \left(2C_{1,\rm inf}+\frac{C_{\rm inf}^2}{2\mu}+1\right)\|\be_r^j\|_0^2
+C(\nu+\mu+1)\sum_{j=1}^n\Delta t\|\nabla \bbeta_h^j\|_0^2.\nonumber
\end{eqnarray}
Denoting by
$$
C_u=2C_{1,\rm inf}+\frac{C_{\rm inf}^2}{2\mu}+1,
$$
assuming $\Delta t C_u\le 1/2$ and applying Gronwall's Lemma \cite[Lemma 5.1]{Hey-RanIV} and \eqref{inv_1} we obtain
\begin{eqnarray}\label{eq:error3_b}
&&\|\be_r^n\|_0^2+\nu\sum_{j=1}^n\Delta t \|\nabla \be_r^j\|_0^2+\mu\sum_{j=1}^n\Delta t\|\nabla \cdot \be_r^j\|_0^2\nonumber\\
&&\quad\le e^{2T C_u}\left(\|\be_r^0\|_0^2+C(\nu+\mu+1)\|S^v\|_0\sum_{k=r+1}^{d_v}\lambda_k\right).
\end{eqnarray}
\begin{remark}
Let us observe that the constants in the error bound \eqref{eq:error3_b} are independent on inverse powers of $\nu$ and also that they only depend on the eigenvalues $\lambda_k$. This means that the error of the POD method, as expected, is essentially the error
of the best approximation to the Galerkin approximation in the POD space, i.e., the error in $P_r^v\bu_h^n$, plus the POD error coming from the tail of eigenvalues.
\end{remark}
\begin{Theorem}
Let $\bu$ be the velocity in the Navier-Stokes equations \eqref{NS}, let $\bu_r$ be the grad-div POD
stabilized approximation defined in \eqref{eq:pod_method2}, and assume that the solution $(\bu,p)$ of \eqref{NS} is regular enough. Then, the
following bound holds
\begin{eqnarray}\label{eq:cota_finalSUPv}
\sum_{j=1}^n\Delta t \|\bu_r^j-\bu^j\|_0^2&\le& 3e^{2T C_u}\left(\|\be_r^0\|_0^2+
C(\nu+\mu+1)\|S^v\|_2\sum_{k=r+1}^{d_v} \lambda_k\right.\nonumber\\
&&\quad \left.+3TC(\bu,p,l)^2(h^{2l}+(\Delta t)^2)+3\sum_{k=r+1}^{d_v}\lambda_k\right).
\end{eqnarray}
\end{Theorem}
\begin{proof}
Applying \eqref{eq:lauso}, from  \eqref{eq:error3_b}, \eqref{eq:cota_pod_0} and \eqref{eq:cota_grad_div}
we easily obtain \eqref{eq:cota_finalSUPv}.
\end{proof}
Following \cite{schneier} we propose a way to compute a POD pressure approximation in this framework. Given a function $p_r\in{\cal W}^r$ we consider the following problem: find $\bw_h\in {\boldsymbol X}_h^l$ such that
\begin{eqnarray}\label{eq:supremizer}
(\nabla \bw_h,\nabla\bv_h)=-(\nabla \cdot \bv_h,p_r),\quad \forall \bv_h \in {\boldsymbol X}_h^l.
\end{eqnarray}
The supremizer enrichment algorithm consists of solving \eqref{eq:supremizer} for each basis function $\psi_k$, $k=1,\cdots,r$ (we are considering to work with an equal number of basis functions for both velocity and pressure). Then,
applying a Gram-Schmidt orthonormalization  procedure to the set of solutions, a set of basis functions is obtained $\left\{\bzeta_k\right\}$, $k=1,\ldots,r$. Denoting by
$$
{\boldsymbol S}^r=<\bzeta_1,\bzeta_2,\ldots,\bzeta_{r}>\subset{\boldsymbol V}_{h,l}^\bot\subset {\boldsymbol X}_h^l,
$$
the following in-sup stability condition holds for the spaces ${\boldsymbol S}^{r}$ and ${\cal W}^r$, see \cite[Lemma 4.2]{schneier}:
\begin{eqnarray}\label{eq:inf_sup_supre}
\beta_{r}= \inf_{\psi\in {\cal W}^r}\sup_{\bzeta\in{{\boldsymbol S}^{r}}}
\frac{(\psi,\nabla \cdot \bzeta)}{\|\bzeta\|_{1}
\|\psi\|_{L^2/{\mathbb R}}} \geq \beta_{\rm{is}},
\end{eqnarray}
where $ \beta_{\rm{is}}$ is the constant in the inf-sup condition \eqref{lbbh}.

Using the space ${\boldsymbol S}^{r}$ a pressure $p_r^n\in {\cal W}^r$ can be computed satisfying for all $\bzeta\in {\boldsymbol S}^{r}$
\begin{eqnarray}\label{eq:pres}
(p_r^n,\nabla \cdot \bzeta)=\left(\frac{\bu_r^{n}-\bu_r^{n-1}}{\Delta t},\bzeta\right)+b_h(\bu_r^n,\bu_r^n,\bzeta)
+\mu(\nabla \cdot \bu_r^n,\nabla \cdot \bzeta)
-(\bff^{n},\bzeta).
\end{eqnarray}
Denoting by
$$
\|\bv\|_{{\boldsymbol S}^{r,*}}=\sup_{\bzeta\in {\boldsymbol S}^{r}}\frac{(\bv,\bzeta)}{\|\nabla \bzeta\|_0},
\quad
\|\bv\|_{{\cal \bV}^{r,*}}=\sup_{\bzeta\in {\cal \bV}^{r}}\frac{(\bv,\bzeta)}{\|\nabla \bzeta\|_0},
$$
by $0\leq\alpha<1$ the constant in the strengthened Cauchy-Schwarz inequality between the spaces ${\cal \bV}^r$ and ${\boldsymbol S}^{r}$,
and by
\begin{eqnarray}\label{cotaCr}
C_r^{H^1}=\left\|\sum_{k=1}^r \nabla \bvar_k\right\|_0,
\end{eqnarray}
and letting $\bvar\in {\cal \bV}^{r}$, it holds (see \cite[Lemma 5.13]{schneier})
\begin{eqnarray}\label{eq:cota_dual}
\|\bvar\|_{{\boldsymbol S}^{r,*}}\le \alpha C_P C_r^{H^1}\|\bvar\|_{{\cal \bV}^{r,*}},
\end{eqnarray}
where $C_P$ is the constant in the Poincar\'e inequality.

Now, we observe that adding and subtracting $P_r^p p_h^n$ from \eqref{eq:gal_grad_div} we get
\begin{eqnarray}\label{eq:p_hres}
(P_r^p p_h^n,\nabla \cdot \bzeta)&=&\left(\frac{\bu_h^{n}-\bu_h^{n-1}}{\Delta t},\bzeta\right)+b_h(\bu_h^n,\bu_h^n,\bzeta)
+\mu(\nabla \cdot \bu_h^n,\nabla \cdot \bzeta)
-(\bff^{n},\bzeta)\nonumber\\
&&\quad+(P_r^p p_h^n-p_h^n,\nabla \cdot \bzeta), \quad \forall\bzeta\in {\boldsymbol S}^{r}.
\end{eqnarray}
Subtracting \eqref{eq:p_hres} from \eqref{eq:pres} and applying \eqref{eq:inf_sup_supre}, \eqref{eq:cota_dual} and \eqref{diver_vol}  we get
\begin{eqnarray}\label{eq:erpre1}
&&\|z_r^n\|_0\le \frac{1}{\beta_{r}}\left(\alpha C_P C_r^{H^1} \left\|\frac{(\bu_r^n-\bu_h^n)-(\bu_r^{n-1}-\bu_h^{n-1})}{\Delta t }\right\|_{{\cal \bV}^{r,*}}\right.\\
&&\quad \left.\sup_{\bzeta\in {\boldsymbol S}^{r}}\frac{b_h(\bu_r^n,\bu_r^n,\bzeta)-b_h(\bu_h^n,\bu_h^n,\bzeta)}{\|\nabla \bzeta\|_0}
+\mu  \|\nabla \cdot(\bu_r^n-\bu_h^n)\|_0+\|\xi_h^n\|_0\right).\nonumber
\end{eqnarray}
We will bound now the terms on the right-hand side of \eqref{eq:erpre1}. For the first one we observe that
\begin{eqnarray*}
\left\|\frac{(\bu_r^n-\bu_h^n)-(\bu_r^{n-1}-\bu_h^{n-1})}{\Delta t }\right\|_{{\cal \bV}^{r,*}}
=\left\|\frac{\be_r^n-\be_r^{n-1}}{\Delta t }\right\|_{{\cal \bV}^{r,*}}.
\end{eqnarray*}
Applying now \eqref{eq:error_need} we obtain
\begin{eqnarray}\label{eq:ertem}
&&\left\|\frac{(\bu_r^n-\bu_h^n)-(\bu_r^{n-1}-\bu_h^{n-1})}{\Delta t }\right\|_{{\cal \bV}^{r,*}}
\le \nu\|\nabla \be_r^n\|_0+\mu\|\nabla\cdot\be_r^n\|_0\nonumber\\
&&\quad+(\nu+\mu) \|\nabla\bbeta_h^n\|_0+\sup_{\bzeta\in {\cal \bV}^{r}}\frac{b_h(\bu_h^n,\bu_h^n,\bzeta)-b_h(\bu_r^n,\bu_r^n,\bzeta)}{\|\nabla \bzeta\|_0}.
\end{eqnarray}
To bound the last term on the right-hand side of \eqref{eq:ertem} we write
\begin{eqnarray*}
&&b_h(\bu_r^n,\bu_r^n,\bzeta)-b_h(\bu_h^n,\bu_h^n,\bzeta)=b_h(\bu_r^n,\bu_r^n-\bu_h^n,\bzeta)+b_h(\bu_r^n-\bu_h^n,\bu_h,\bzeta)
\nonumber\\
&&\quad=b_h(\bu_r^n,\bzeta,\bu_r^n-\bu_h^n)+b_h(\bu_r^n-\bu_h^n,\bu_h,\bzeta)\nonumber\\
&&\quad\le\|\bu_r^n-\bu_h^n\|_0\left(\|\bu_r^n\|_{L^\infty}\|\bzeta\|_1+\frac{1}{2}\|\nabla \cdot\bu_r^n\|_{L^{2d/(d-1)}}\|\bzeta\|_{L^{2d}} \right)
\nonumber\\
&&\quad+\frac{1}{2}\left((\bu_r^n-\bu_h^n)\cdot \nabla \bu_h,\bzeta\right)-\frac{1}{2}\left((\bu_r^n-\bu_h^n)\cdot \nabla\bzeta,
\bu_h^n\right)\nonumber\\
&&\quad\le\|\bu_r^n-\bu_h^n\|_0\left(\|\bu_r^n\|_{L^\infty}\|\bzeta\|_1+\frac{1}{2}\|\nabla \cdot\bu_r^n\|_{L^{2d/(d-1)}}\|\bzeta\|_{L^{2d}} \right)
\nonumber\\
&&\quad+\frac{1}{2}\|\bu_r^n-\bu_h^n\|_0\left(\|\nabla \bu_h\|_{L^{2d/(d-1)}}\|\bzeta\|_{L^{2d}}+\|\bu_h^n\|_{L^\infty}\|\bzeta\|_1\right).
\end{eqnarray*}
Now, we observe that applying \eqref{inv}, \eqref{eq:cotaPlinf2}
and \eqref{eq:error3_b} we obtain
\begin{eqnarray*}
&&\|\bu_r^n\|_{L^\infty}\le \|\be_r^n \|_{L^\infty}+\|P_r^v \bu_h^n\|_{L^\infty}\le C h^{-d/2}\|\be_r^n\|_0+C_{\rm inf}\nonumber\\
&&\le C_{\bu_r,\rm inf}:=C h^{-d/2}e^{T C_u}\left(\|\be_r^0\|_0^2+C(\nu+\mu+1)\|S^v\|_0\sum_{k=r+1}^{d_v}\lambda_k\right)^{1/2}+C_{\rm inf}.
\end{eqnarray*}
Applying exactly the same argument we also obtain
\begin{eqnarray*}
&&\|\nabla \cdot\bu_r^n\|_{L^{2d/(d-1)}}\le  C h^{-1/2}\|\be_r^n\|_1+C_{\rm ld}\nonumber\\
&&\le C_{\bu_r,\rm ld}:=C h^{-d/2}e^{T C_u}\|S^v\|_2^{1/2}\left(\|\be_r^0\|_0^2+C(\nu+\mu+1)\|S^v\|_0\sum_{k=r+1}^{d_v}\lambda_k\right)^{1/2}
\nonumber\\
&&\quad+C_{\rm ld}.
\end{eqnarray*}
With the above error bounds, Sobolev embeddings \eqref{sob1}, \eqref{eq:uh_infty} and \eqref{eq:grad_uh_infty} we finally obtain
\begin{eqnarray}\label{eq:nonli}
&&b_h(\bu_r^n,\bu_r^n,\bzeta)-b_h(\bu_h^n,\bu_h^n,\bzeta)
\le C\|\bu_r^n-\bu_h^n\|_0\|\bzeta\|_1
\end{eqnarray}
Going back to \eqref{eq:ertem} we reach
\begin{eqnarray*}
\left\|\frac{(\bu_r^n-\bu_h^n)-(\bu_r^{n-1}-\bu_h^{n-1})}{\Delta t }\right\|_{{\cal \bV}^{r,*}}
&\le& \nu\|\nabla \be_r^n\|_0+\mu\|\nabla\cdot\be_r^n\|_0+(\nu+\mu) \|\nabla\bbeta_h^n\|_0\nonumber\\
&&\quad+C\|\bu_r^n-\bu_h^n\|_0
\end{eqnarray*}
Taking into account that
\begin{eqnarray}\label{eq:otra_norma0}
\|\bu_r^n-\bu_h^n\|_0\le \|\be_r^n\|_0+\|P_r^v\bu_h^n-\bu_h^n\|_0,
\end{eqnarray}
we finally obtain
\begin{eqnarray}\label{eq:ertem2}
\left\|\frac{(\bu_r^n-\bu_h^n)-(\bu_r^{n-1}-\bu_h^{n-1})}{\Delta t }\right\|_{{\cal \bV}^{r,*}}
&\le& \nu\|\nabla \be_r^n\|_0+\mu\|\nabla\cdot\be_r^n\|_0+(\nu+\mu) \|\nabla\bbeta_h^n\|_0
\nonumber\\
&&\quad+\|\be_r^n\|_0+\|P_r^v\bu_h^n-\bu_h^n\|_0.
\end{eqnarray}
For the second term on the right-hand side of \eqref{eq:erpre1} we apply \eqref{eq:nonli} and \eqref{eq:otra_norma0}. Inserting also \eqref{eq:ertem2}
into \eqref{eq:erpre1} we finally reach
\begin{eqnarray}\label{eq:erpre2}
\|z_r^n\|_0&\le& \beta_{r}^{-1}\alpha C_P C_r^{H^1}\left( \nu\|\nabla \be_r^n\|_0+\mu\|\nabla\cdot\be_r^n\|_0+(\nu+\mu) \|\nabla\bbeta_h^n\|_0
\right)\nonumber\\
&&\quad+\beta_{r}^{-1}\alpha C_P C_r^{H^1}\left(\|\be_r^n\|_0+\|P_r^v\bu_h^n-\bu_h^n\|_0 \right)\\
&&\quad C\beta_{r}^{-1}\left(\|\be_r^n\|_0+\|P_r^v\bu_h^n-\bu_h^n\|_0+\mu  \|\nabla \cdot(\bu_r^n-\bu_h^n)\|_0+\|\xi_h^n\|_0\right)
.\nonumber
\end{eqnarray}
We observe that
\begin{eqnarray*}
\|\nabla \cdot(\bu_r^n-\bu_h^n)\|_0\le \|\nabla \cdot\be_r^n\|_0+\|\nabla (P_r^v\bu_h^n-\bu_h^n)\|_0.
\end{eqnarray*}
From \eqref{eq:erpre2} 
we get
\begin{eqnarray}\label{eq:erpre3}
\sum_{j=1}^n\Delta t \|z_r^j\|_0^2&\le& C \alpha C_r^{H^1}\sum_{j=1}^n\Delta t\left( \nu\|\nabla \be_r^j\|_0^2+\mu\|\nabla\cdot\be_r^j\|_0^2\right)
\nonumber\\
&&\quad +C \alpha C_r^{H^1}(\nu+\mu)\sum_{j=1}^n\|\nabla(P_r^v\bu_h^j-\bu_h^j)\|_0^2\nonumber\\
&&\quad +C\left(1+\alpha C_r^{H^1}\right)\sum_{j=1}^n\Delta t \left(\|\be_r^j\|_0^2+\|P_r^v\bu_h^j-\bu_h^j\|_0^2\right)\nonumber\\
&&\quad + C \mu\sum_{j=1}^n\Delta t \left(\|\nabla\cdot\be_r^j\|_0^2+\|\nabla(P_r^v\bu_h^j-\bu_h^j)\|_0^2\right)\nonumber\\
&&\quad +C \sum_{j=1}^n\Delta t \|P_r^p p_h^j-p_h^j\|_0^2.
\end{eqnarray}
To conclude we apply \eqref{eq:cota_pod_0}, \eqref{eq:cota_pod_0_pre}, \eqref{inv_1} and \eqref{eq:error3_b} to \eqref{eq:erpre3}
\begin{eqnarray}\label{eq:erpre4}
&&\sum_{j=1}^n\Delta t \|z_r^j\|_0^2\nonumber\\
&&\le C(1+T)\left(1+\alpha C_r^{H^1}\right)e^{2T C_u}\left(\|\be_r^0\|_0^2+C(\nu+\mu+1)\|S^v\|_0\sum_{k=r+1}^{d_v}\lambda_k\right)
\nonumber\\
&&\quad+\left(C \alpha C_r^{H^1}(\nu+\mu)+\mu\right)\|S^v\|_2\sum_{k=r+1}^{d_v} \lambda_k\nonumber\\
&&\quad+C\left(1+\alpha C_r^{H^1}\right)\sum_{k=r+1}^{d_v} \lambda_k+C\sum_{k=r+1}^{d_p} \gamma_k.
\end{eqnarray}
\begin{Theorem}
Let $p$ be the pressure in the Navier-Stokes equations \eqref{NS}, let $p_r$ be the grad-div POD
stabilized pressure defined in \eqref{eq:pres}, and assume that the solution $(\bu,p)$ of \eqref{NS} is regular enough. Then, the
following bound holds
\begin{eqnarray}\label{eq:erpre4_ul}
&&\sum_{j=1}^n\Delta t \|p_r^j-p^j\|_0^2\nonumber\\
&&\le C(1+T)\left(1+\alpha C_r^{H^1}\right)e^{2T C_u}\left(\|\be_r^0\|_0^2+C(\nu+\mu+1)\|S^v\|_0\sum_{k=r+1}^{d_v}\lambda_k\right)
\nonumber\\
&&\quad+3\left(C \alpha C_r^{H^1}(\nu+\mu)+\mu\right)\|S^v\|_2\sum_{k=r+1}^{d_v} \lambda_k\nonumber\\
&&\quad+C\left(1+\alpha C_r^{H^1}\right)\sum_{k=r+1}^{d_v} \lambda_k+C\sum_{k=r+1}^{d_p} \gamma_k+3C(\bu,p,l)^2 (h^{l}+\Delta t)^2.
\end{eqnarray}
\end{Theorem}
\begin{proof}
Since applying triangle inequality we have
\begin{equation}\label{eq:cota_finalSUPp}
\sum_{j=1}^n\Delta t \|p_r^j-p^j\|_0^2\le 3\sum_{j=1}^n\Delta t \|z_r^j\|_0^2+3
\sum_{j=1}^n\Delta t \|P_h p_h^j-p_h^j\|_0^2+3\sum_{j=1}^n\Delta t \|p_h^j-p^j\|_0^2,
\end{equation}
the bound \eqref{eq:erpre4_ul}
 for the pressure follows from \eqref{eq:erpre4}, \eqref{eq:cota_pod_0_pre} and \eqref{eq:cota_grad_div_pre}.
\end{proof}

\section{Numerical experiments}\label{sec:num}

In this section we present numerical results for the LPS-ROM \eqref{eq:pod_method1} and the grad-div-ROM \eqref{eq:pod_method2}-\eqref{eq:pres}, introduced and analyzed in the previous sections. For the LPS-ROM \eqref{eq:pod_method1}, the standard discrete inf-sup condition is circumvented and POD modes, which are computed by the LPS-FEM \eqref{eq:gal_est}, are not required indeed to be weakly divergence-free. For the grad-div-ROM \eqref{eq:pod_method2}-\eqref{eq:pres}, the standard discrete inf-sup condition is recovered through supremizer enrichment \cite{Rozza15,RozzaVeroy07} and at least weakly divergence-free POD modes are required, which are computed by the grad-div-FEM \eqref{eq:gal_grad_div}. The numerical experiments are performed on the benchmark problem of the 2D unsteady flow around a cylinder with circular cross-section \cite{SchaferTurek96} at Reynolds number $Re=100$. The open-source FE software FreeFEM \cite{Hecht12} has been used to run the numerical experiments. 

Within this framework, we also propose an adaptive in time algorithm for the online grad-div parameter $\mu$ used both in the LPS-ROM \eqref{eq:pod_method1} and the grad-div-ROM \eqref{eq:pod_method2}-\eqref{eq:pres}, by adjusting dissipation arising from the online grad-div stabilization term in order to better match the FOM energy, which proves to significantly improve the long time ROM accuracy.

\medskip

{\em Setup for numerical simulations.}
Following \cite{SchaferTurek96}, the computational domain is given by a rectangular channel with a circular hole (see Figure \ref{fig:Mesh} for the computational grid used):
$$
\Om=\{(0,2.2)\times(0,0.41)\}\backslash \{\xv : (\xv-(0.2,0.2))^2 \leq 0.05^2\}.
$$
\begin{figure}[htb]
\begin{center}
\centerline{\includegraphics[width=4.75in]{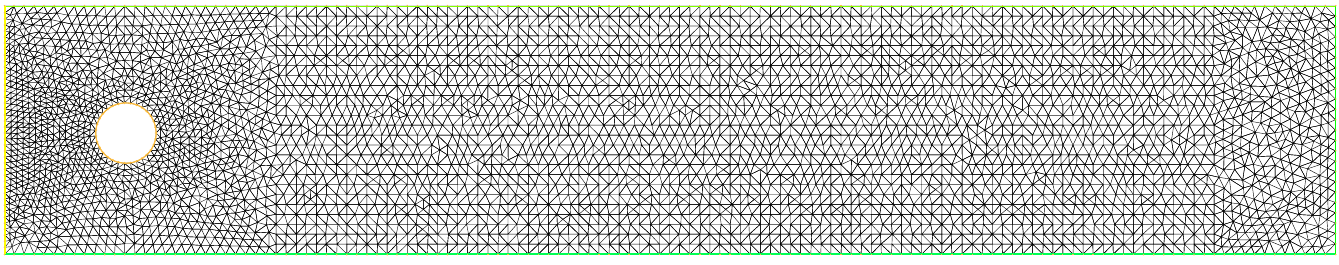}}
\caption{Computational grid.}\label{fig:Mesh}
\end{center}
\end{figure}

No slip boundary conditions are prescribed on the horizontal walls and on the cylinder, and a parabolic inflow profile is provided at the inlet:
$$
\uv(0,y,t)=(4U_{m}y(A-y)/A^2, 0)^{T},
$$
with $U_{m}=\uv(0,H/2,t)=1.5\,\rm{m/s}$, and $A=0.41\,\rm{m}$ the channel height. At the outlet, we impose outflow (do nothing) boundary conditions $(\nu\nabla\uv - p\,Id)\nv={\bf 0}$, with $\nv$ the outward normal to the domain. The kinematic viscosity of the fluid is $\nu=10^{-3}\,\rm{m^{2}/s}$ and there is no external (gravity) forcing, i.e. $\fv={\bf 0}\,\rm{m/s^2}$. Based on the mean inflow velocity $\overline{U}=2U_{m}/3=1\,\rm{m/s}$, the cylinder diameter $D=0.1\,\rm{m}$ and the kinematic viscosity of the fluid $\nu=10^{-3}\,\rm{m^{2}/s}$, the Reynolds number considered is $Re=\overline{U}D/\nu=100$. In the fully developed periodic regime, a vortex shedding can be observed behind the obstacle, resulting in the well-known von K\'arm\'an vortex street (see Figure \ref{fig:FinFOMSol}).

\begin{figure}[htb]
\begin{center}
\includegraphics[width=2.38in]{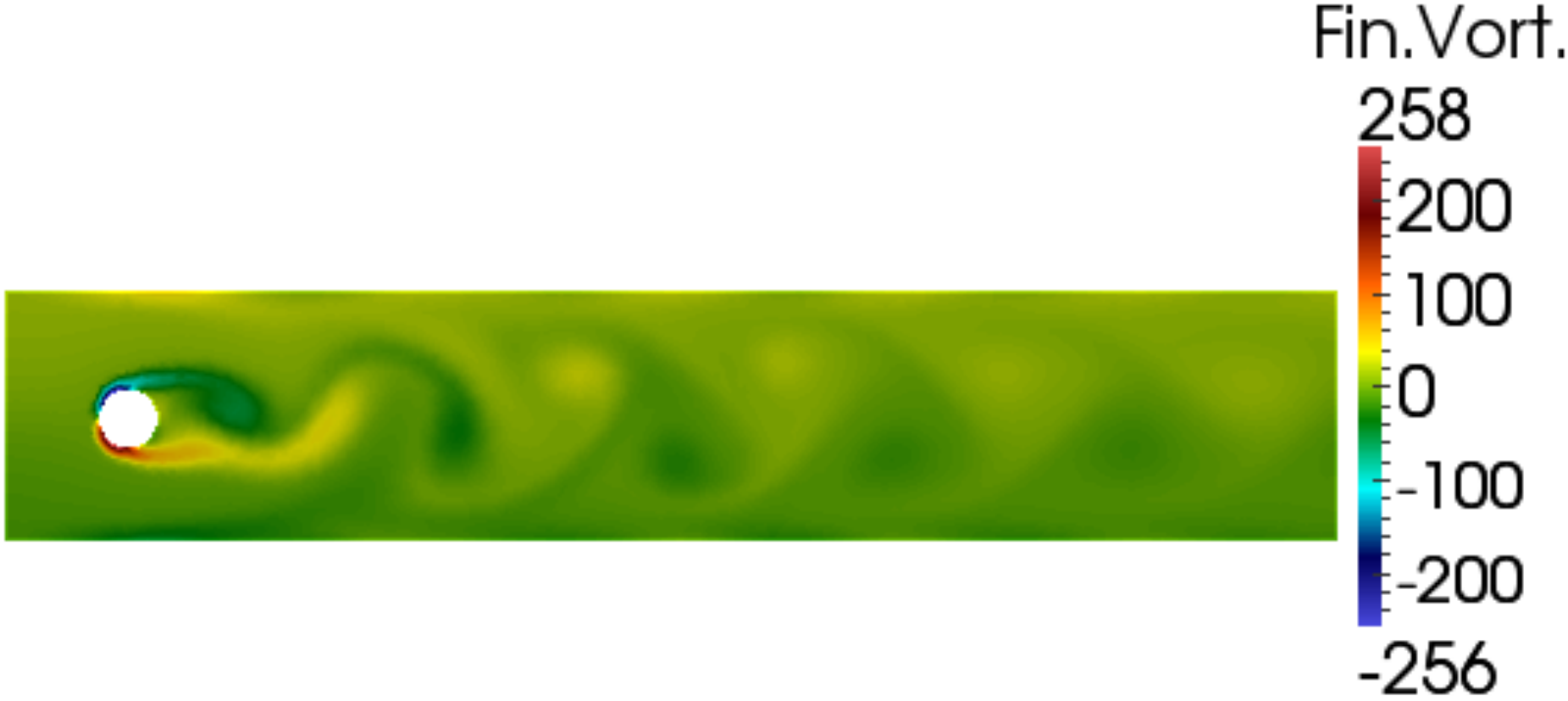}\hspace{-0.1cm}
\includegraphics[width=2.38in]{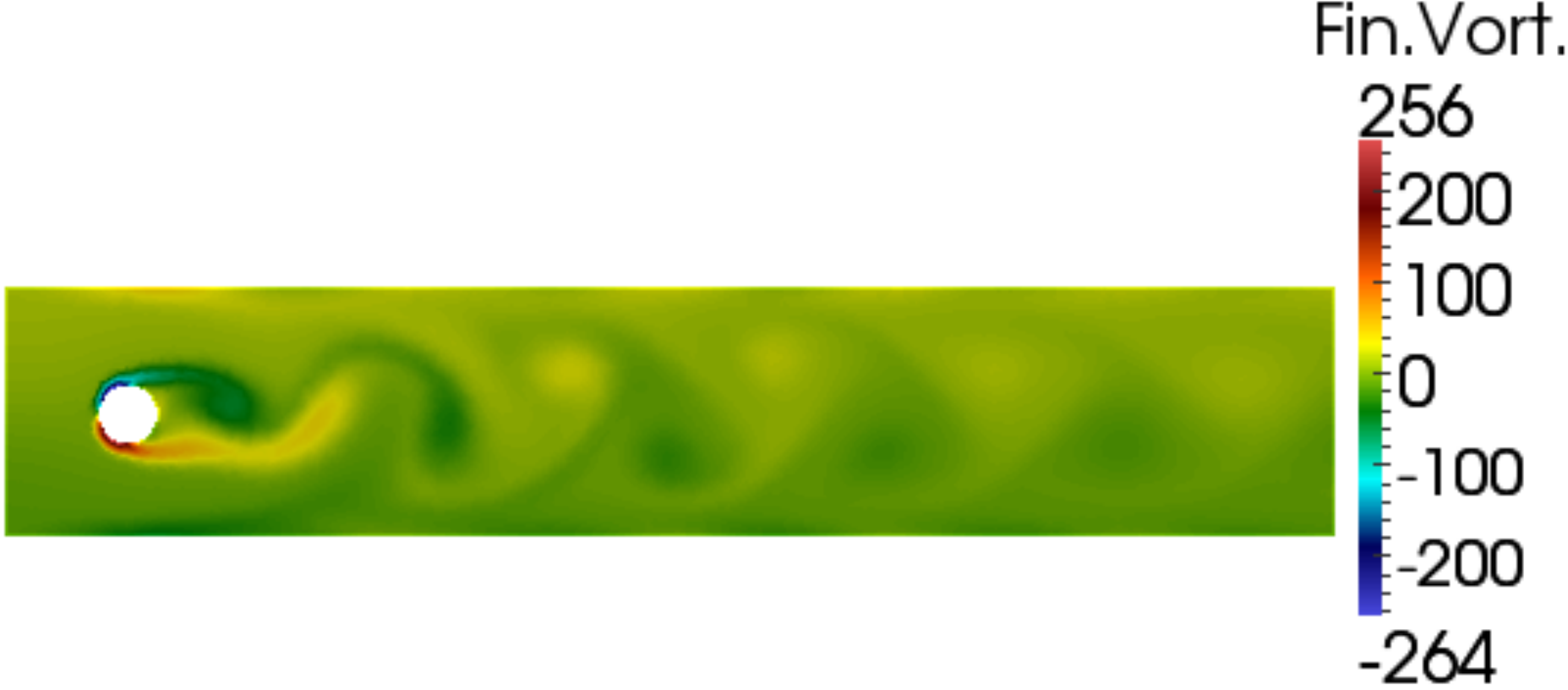}
\caption{Final FOM vorticity for LPS-FEM \eqref{eq:gal_est} (left) and grad-div-FEM \eqref{eq:gal_grad_div} (right).}\label{fig:FinFOMSol}
\end{center}
\end{figure}

For the evaluation of computational results, we are interested in studying the temporal evolution of the following
quantities of interest. The kinetic energy of the flow is the most frequently monitored quantity, given by:
$$
E_{kin}=\frac{1}{2}\nor{\uv}{{\bf L}^2}^2.
$$
Other relevant quantities of interest are the drag and lift coefficients. In order to reduce the boundary approximation influences, in the present work these quantities are computed as volume integrals \cite{John04paper}:
$$
c_{D}=-\frac{2}{D\overline{U}^2}\left[(\partial_t \uv,\vv_D)+b(\uv,\uv,\vv_D)+\nu(\nabla\uv,\nabla\vv_D)-(p,\div\vv_D)\right],
$$
$$
c_{L}=-\frac{2}{D\overline{U}^2}\left[(\partial_t \uv,\vv_L)+b(\uv,\uv,\vv_L)+\nu(\nabla\uv,\nabla\vv_L)-(p,\div\vv_L)\right],
$$
for arbitrary test functions $\vv_{D},\vv_{L}\in {\bf H}^{1}$ such that $\vv_{D}=(1,0)^{T}$ on the boundary of the cylinder and vanishes on the other boundaries, $\vv_{L}=(0,1)^{T}$ on the boundary of the cylinder and vanishes on the other boundaries. In the actual computations, we have used the approach described in \cite{JohnMatthies01} to fix the test functions $\vv_{D},\vv_{L}$ and evaluate the drag and lift coefficients $c_{D},c_{L}$. 
%
%

\medskip 

{\em FOM and POD modes.}
The numerical method used to compute the snapshots for the LPS-ROM \eqref{eq:pod_method1} is the LPS-FEM \eqref{eq:gal_est} described in Section \ref{sec:PN}, with a spatial discretization using equal order ${\bf P}^{2}-\mathbb{P}^2$ FE for the pair velocity-pressure on the relatively coarse computational grid in Figure \ref{fig:Mesh}, for which $h = 2.76\cdot 10^{-2}\,\rm{m}$, resulting in $32\,488$ d.o.f. for velocities and $16\,244$ d.o.f. for pressure. The numerical method used to compute the snapshots for the grad-div-ROM \eqref{eq:pod_method2}-\eqref{eq:pres} is the grad-div-FEM \eqref{eq:gal_grad_div} described in Section \ref{sec:grad-div-ROM}, with a spatial discretization using the mixed inf-sup stable ${\bf P}^{2}-\mathbb{P}^1$ Taylor--Hood FE for the pair velocity-pressure on the same relatively coarse computational grid in Figure \ref{fig:Mesh}, resulting in $32\,488$ d.o.f. for velocities and $4\,151$ d.o.f. for pressure. For the LPS-FEM \eqref{eq:gal_est}, we have used the following expressions for the stabilization parameters: $\tau_{\nu,K}=C_{v}h_{K}$ and $\tau_{p,K}=C_{p}h_{K}$, with $C_{v}=10^{-2}\,\rm{m/s}$ and $C_{p}=10^{-2}\,\rm{m^{-1}s}$, respectively. The expression for $\tau_{\nu,K}$ has been taken from \cite{nos_lps} and in view of conditions \eqref{eq:tau_p2}-\eqref{eq:mu_2} we have chosen a similar expression for $\tau_{p,K}$. In this case, we have $\tau_{\nu,K}\leq 2.76\cdot 10^{-4}\,\rm{m^2/s}$ and $\tau_{p,K}\leq 2.76\cdot 10^{-4}\,\rm{s}$. Following again \cite{nos_lps}, for the grad-div-FEM \eqref{eq:gal_grad_div} we have considered $\mu=\mu_{K}=C_{d}h_{K}$ with $C_{d}=1\,\rm{m/s}$, so that we have $\mu=\mu_{K}\leq 2.76\cdot 10^{-2}\,\rm{m^{2}/s}$.

For the time discretization, in both cases a semi-implicit Backward Differentiation Formula of order two (BDF2) has been applied, which guarantees a good balance between numerical accuracy and computational complexity \cite{AhmedRubino19}. In particular, we have considered an extrapolation for the convection velocity by means of Newton--Gregory backward polynomials \cite{Cellier91}. Without entering into the details of the derivation, for which we refer the reader to e.g. \cite{Cellier91}, we consider the following extrapolation of order two for the discrete velocity: $\widehat{\uv}_{h}^{n}=2\uhv^{n}-\uhv^{n-1}$, $n\geq 1$, in order to achieve a second-order accuracy in time. For the initialization $(n=0)$, we have considered $\uhv^{-1}=\uhv^{0}=\uv_{0h}$, being $\uv_{0h}$ the initial condition, so that the time scheme reduces to the semi-implicit Euler method for the first time step $(\Delta t)^{0}=(2/3)\Delta t$. In both FOM, an impulsive start is performed, i.e. the initial condition is a zero velocity field, and the time step is $\Delta t = 2\times 10^{-3}\,\rm{s}$. Time integration is performed till a final time $T=7\,\rm{s}$. In the time period $[0,5]\,\rm s$, after an initial spin-up, the flow is expected to develop to full extent, including a subsequent relaxation time. Afterwards, it reaches a periodic-in-time (statistically- or quasi-steady) state.

In Figure \ref{fig:QOIFOM1} (left), we plot drag coefficient temporal evolution for the FOM solution computed with LPS-FEM \eqref{eq:gal_est} and grad-div-FEM \eqref{eq:gal_grad_div}. Results are very close and, despite the relatively coarse computational grid used, they agree quite well with reference values from \cite{SchaferTurek96}.
\begin{figure}[htb]
\begin{center}
\includegraphics[width=2.5in]{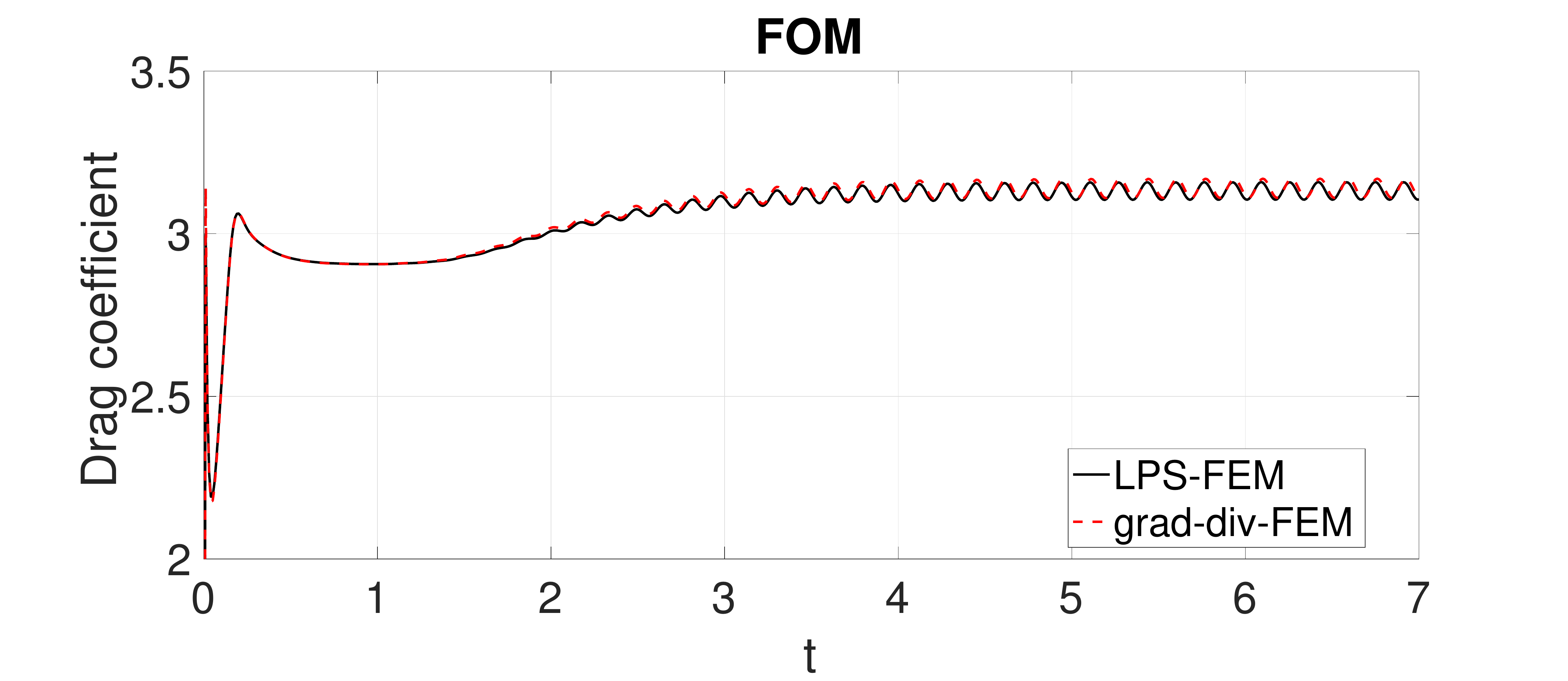}
\includegraphics[width=2.5in]{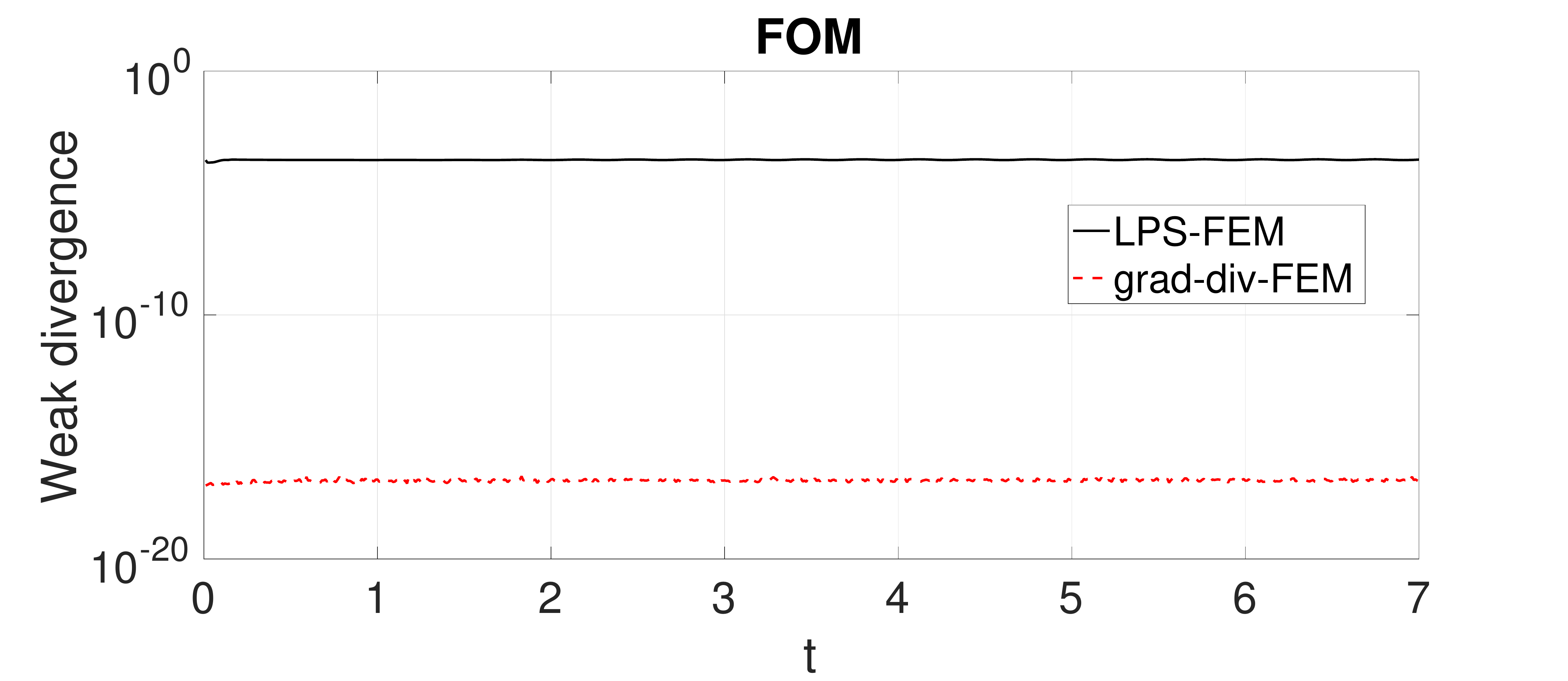}
\caption{Temporal evolution of drag coefficient (left) and weak divergence (right) 
for the FOM solution computed with LPS-FEM \eqref{eq:gal_est} and grad-div-FEM \eqref{eq:gal_grad_div}.}\label{fig:QOIFOM1}
\end{center}
\end{figure}
 
In Figure \ref{fig:QOIFOM1} (right), we show the weak divergence temporal evolution for the FOM solution computed with LPS-FEM \eqref{eq:gal_est} and grad-div-FEM \eqref{eq:gal_grad_div}, obtained by plotting $\max_{q_{h}\in Q_h} |(\nabla\cdot \uhv^n,q_h)|$, with $q_h$ the FE test functions varying in $Q_h=\mathbb{P}^2$ for LPS-FEM \eqref{eq:gal_est} and in $Q_h=\mathbb{P}^1$ for grad-div-FEM \eqref{eq:gal_grad_div}. From this figure, it is evident that the snapshots computed with LPS-FEM \eqref{eq:gal_est} are no weakly divergence-free, while the snapshots computed with grad-div-FEM \eqref{eq:gal_grad_div} are weakly divergence-free (up to machine precision).
%

The POD modes are generated in $L^2$ by the method of snapshots with velocity centered-trajectories \cite{IliescuJohn15} by storing every FOM solution from $t=5$, when the solution had reached a periodic-in-time state, and using one period of snapshot data. The full period length of the statistically steady state is $0.332\,\rm{s}$, thus we collect $167$ snapshots for both velocity and pressure. The rank of the velocity data set is $d_{v}=25$ for LPS-FEM \eqref{eq:gal_est} and $d_{v}=27$ for grad-div-FEM \eqref{eq:gal_grad_div} (for which $\lambda_k<10^{-10}, k>d_{v}$), while the rank of the pressure data set is $d_{p}=23$ for both LPS-FEM \eqref{eq:gal_est} and grad-div-FEM \eqref{eq:gal_grad_div} (for which $\gamma_k<10^{-10}, k>d_{p}$). 
Figure \ref{fig:PODstat} shows the decay of POD velocity ($\lambda_k$, $k=1,\ldots, d_v$) and pressure ($\gamma_k$, $k=1,\ldots, d_p$) eigenvalues (left) computed with LPS-FEM \eqref{eq:gal_est} and grad-div-FEM \eqref{eq:gal_grad_div}, together with the corresponding captured system's energy (right), given by $100\sum_{k=1}^{r}\lambda_{k}/\sum_{k=1}^{d_v}\lambda_{k}$ for velocity and $100\sum_{k=1}^{r}\gamma_{k}/\sum_{k=1}^{d_p}\gamma_{k}$ for pressure. Note that the first $r=5$ POD modes already capture more than 99\% of the system's velocity-pressure energy.

\begin{figure}[htb]
\begin{center}
\includegraphics[width=2.5in]{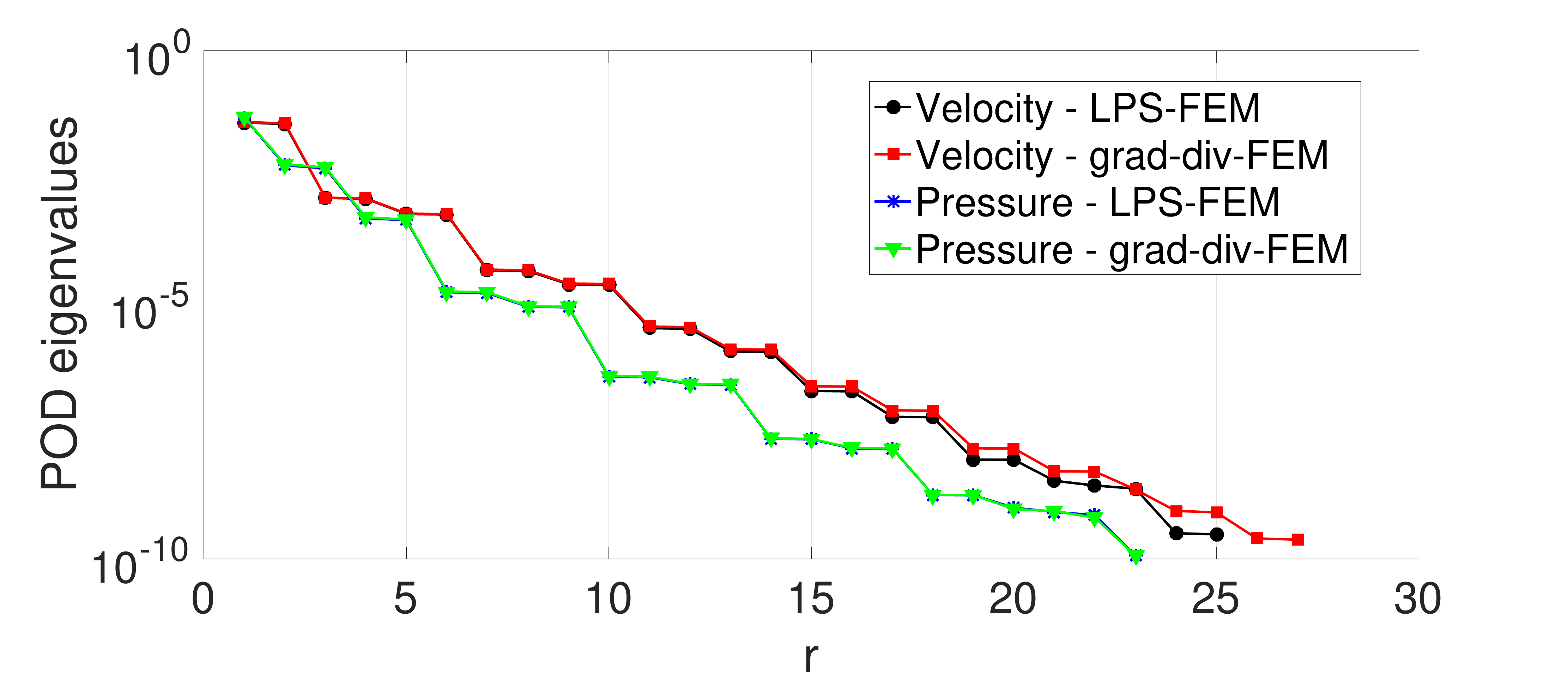}
\includegraphics[width=2.5in]{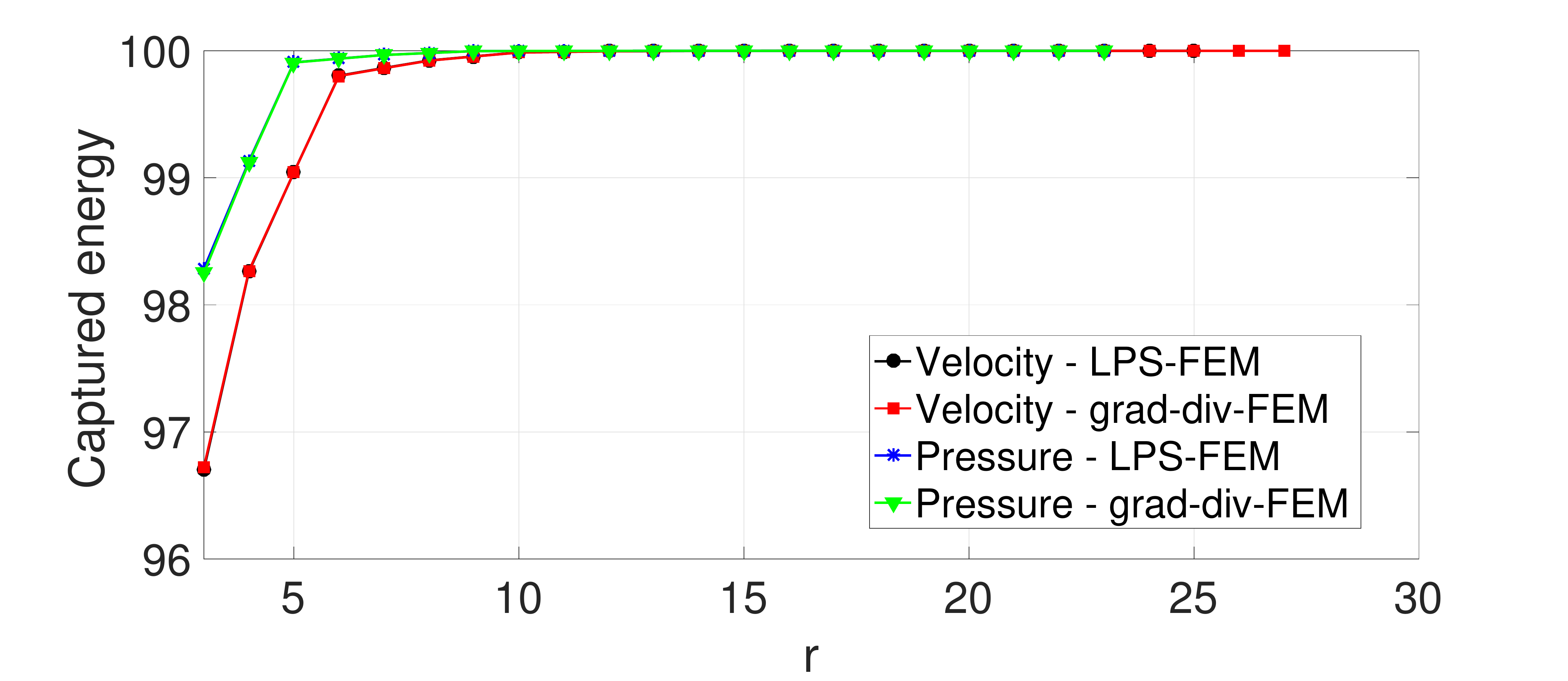}
\caption{POD velocity-pressure eigenvalues (left) and captured system's velocity-pressure energy (right) 
computed with LPS-FEM \eqref{eq:gal_est} and grad-div-FEM \eqref{eq:gal_grad_div}.}\label{fig:PODstat}
\end{center}
\end{figure}

With POD modes generated, the fully discrete LPS-ROM \eqref{eq:pod_method1} and grad-div-ROM \eqref{eq:pod_method2}-\eqref{eq:pres} are constructed as discussed in the previous sections, using the semi-implicit BDF2 time scheme as for the corresponding FOM, and run first in the time interval $[5,7]\,\rm{s}$ with $\Delta t=2\times 10^{-3}\,\rm{s}$ and a small number ($r=8$) of POD velocity-pressure modes, which already gives a reasonable accuracy for the proposed methods. The initial reduced-order velocity is given by the $L^2$-orthogonal projection of the corresponding velocity snapshot at $t=5\,\rm{s}$ on the respective POD velocity space. Note that to recover the online pressure for the grad-div-ROM \eqref{eq:pod_method2}-\eqref{eq:pres}, the same number ($r=8$) of POD supremizers modes has been used. Also, for the LPS-ROM \eqref{eq:pod_method1}, we have considered the same stabilization parameters $\tau_{\nu,K}, \tau_{p,K}$ used in the LPS-FEM \eqref{eq:gal_est} computation. However, a different treatment has been done to the online grad-div stabilization parameter $\mu$. Indeed, we have found convenient to consider an adaptive in time algorithm for the online grad-div parameter $\mu$ used both in the LPS-ROM \eqref{eq:pod_method1} and the grad-div-ROM \eqref{eq:pod_method2}-\eqref{eq:pres}, by adjusting dissipation arising from the online grad-div stabilization term in order to better match the FOM energy, which demonstrated to significantly improve the long time ROM accuracy.

\medskip 

{\em Adaptive in time algorithm for online grad-div parameter $\mu$.}
Hereafter, we describe the adaptive in time algorithm proposed and numerically investigated for the computation of the grad-div parameter $\mu$ in the online phase. This algorithm takes inspiration from the one used in \cite{zerfas_et_al} to compute the nudging parameter in order to further improve the long time accuracy of a data assimilation ROM. To our knowledge, this is the first time it is applied in a grad-div stabilization ROM framework. Instead of considering the same grad-div parameter $\mu=\mu_{K}$ used in the grad-div-FEM \eqref{eq:gal_grad_div}, we propose to perform a comparison within a constant and an adaptive in time $\mu$, based on the accuracy of the energy prediction of both the LPS-ROM \eqref{eq:pod_method1} and the grad-div-ROM \eqref{eq:pod_method2}-\eqref{eq:pres}. The adaptive in time strategy consists in adjusting $\mu$ so that the contribution of the online grad-div stabilization term removes dissipation if the ROM energy is too small, and adds dissipation if the energy is too large with respect to the FOM energy. First of all, we run the ROM for a constant $\mu=\bar{\mu}$, fixed minimizing the $L^{\infty}$ error in time with respect to the snapshots energy computed in one period and then repeated in the rest of periods, thus being the snapshots data to construct the reduced basis sufficient to compute the constant $\bar{\mu}$, and no further information is needed. The same holds for the adaptive in time algorithm detailed below.

\medskip

{\bf Algorithm} {\em (ROM with adaptive in time grad-div parameter $\mu$)}.
\begin{enumerate}
\item Initialize the online grad-div parameter $\mu=\bar{\mu}$.
\item Set $\mu_{min}>0$ the minimum value that $\mu$ can reach in the algorithm. 
\item Set $F$ to be the frequency in number of time steps to adapt $\mu$.
\item Set $\delta$ to be the adjustment size to change $\mu$.
\item Set $tol$ to be the tolerance chosen for making a change to $\mu$.
\item For time step $n=1,2,\ldots$

if $mod(n,F)==0$
\begin{itemize}
\item Compute $E_{kin}^{diff}=\displaystyle\frac{1}{2}\left(\nor{\uv_{r}^{n}}{{\bf L}^2}^2 - \nor{\uv_{h}^{mod(n,M)}}{{\bf L}^2}^2\right)$, where $M$ is the number of snapshots collected.
\item if $E_{kin}^{diff}>tol$, set $\mu=\max\{\mu_{min},\mu+\delta\}$, 

else if $E_{kin}^{diff}<-tol$, set $\mu=\max\{\mu_{min},\mu-\delta\}$. 
\end{itemize}
end

Recompute $\uv_{r}^{n}$.
\end{enumerate}

\medskip

{\em Numerical results.}
To assess the numerical accuracy of the proposed methods LPS-ROM \eqref{eq:pod_method1} and grad-div-ROM \eqref{eq:pod_method2}-\eqref{eq:pres}, the temporal evolution of the drag and lift coefficients, and kinetic energy are monitored and compared to the corresponding FOM solutions in the time interval $[5,7]\,\rm{s}$, corresponding to six period for the lift coefficients. Thus, we are actually testing the ability of the considered ROM to predict/extrapolate in time, monitoring their performance over a six times larger time interval with respect to the one used to compute the snapshots and generate the POD modes. In particular, for both methods, we perform a comparison within a constant and an adaptive in time $\mu$, following the strategy described above.

We start by reporting numerical results for the LPS-ROM \eqref{eq:pod_method1}. Numerical results for drag and lift predictions using $r=8$ velocity-pressure modes are shown in Figure \ref{fig:QOIPOD1LPS}, where we display a comparison within LPS-FEM \eqref{eq:gal_est} and LPS-ROM \eqref{eq:pod_method1} with constant $\mu=\bar{\mu}=2.4$ and adaptive $\mu$ (starting $\mu=\bar{\mu}=2.4$, $\mu_{min}=10^{-1}$, $F=5$, $\delta=10^{-1}$, $tol=10^{-3}$). From this figure, we observe that the LPS-ROM \eqref{eq:pod_method1} allows to compute rather accurate quantities of interest. Indeed, the temporal evolution of the drag and lift coefficient is rather close to that of the LPS-FEM \eqref{eq:gal_est}, being the drag coefficient temporal evolution the most sensitive quantity presenting larger differences (up to $3.5\%$). Note that along the time interval $[5,7]\,\rm{s}$ results are almost similar using constant and adaptive $\mu$. In Figure \ref{fig:QOIPOD2LPS}, we show on the left the temporal evolution of absolute error in kinetic energy $|E_{kin,r}-E_{kin,h}|$, 
and on the right the corresponding temporal evolution of the adaptive grad-div coefficient $\mu$. Note that along the time interval $[5,7]\,\rm{s}$ the kinetic energy error levels are quite similar using constant and adaptive $\mu$, and in both cases they are maintained below $3\cdot 10^{-3}$. 

\begin{figure}[htb]
\begin{center}
\includegraphics[width=2.5in]{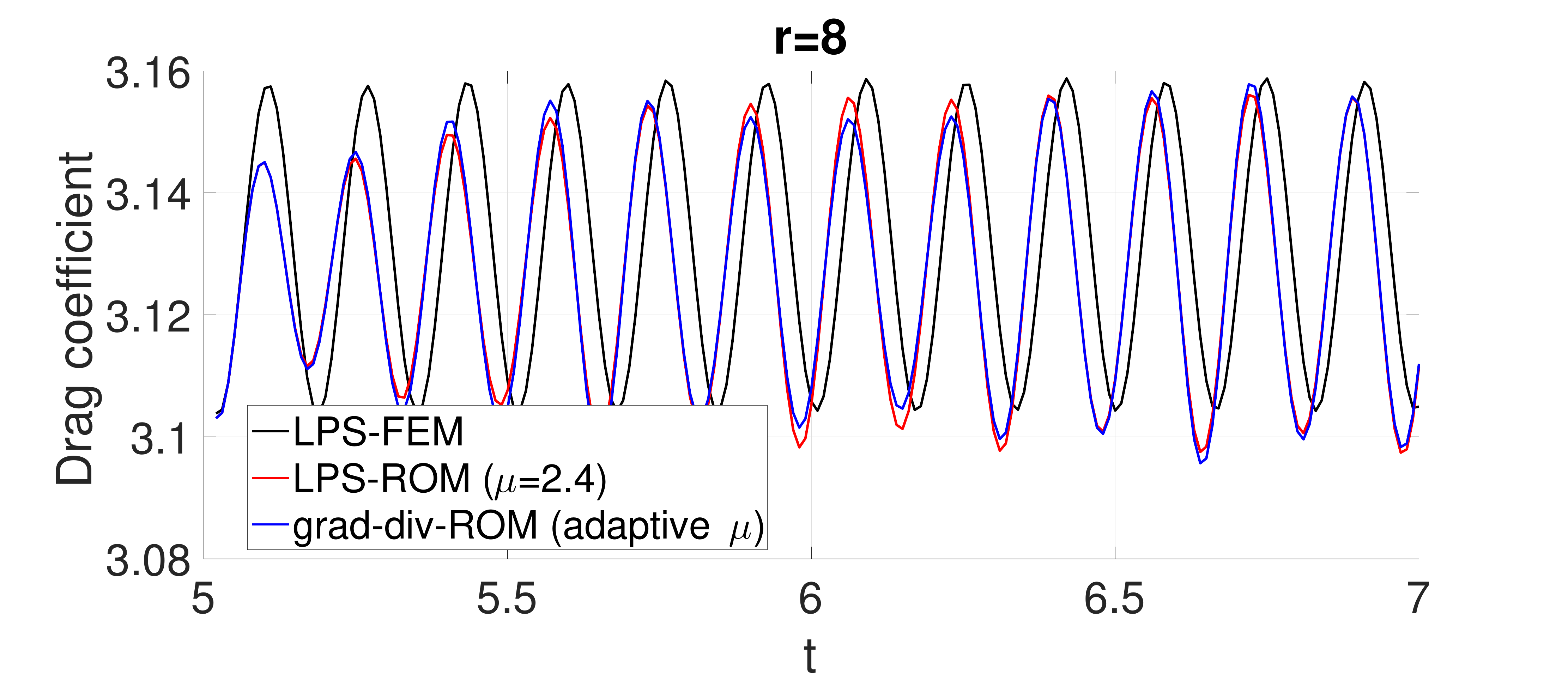}
\includegraphics[width=2.5in]{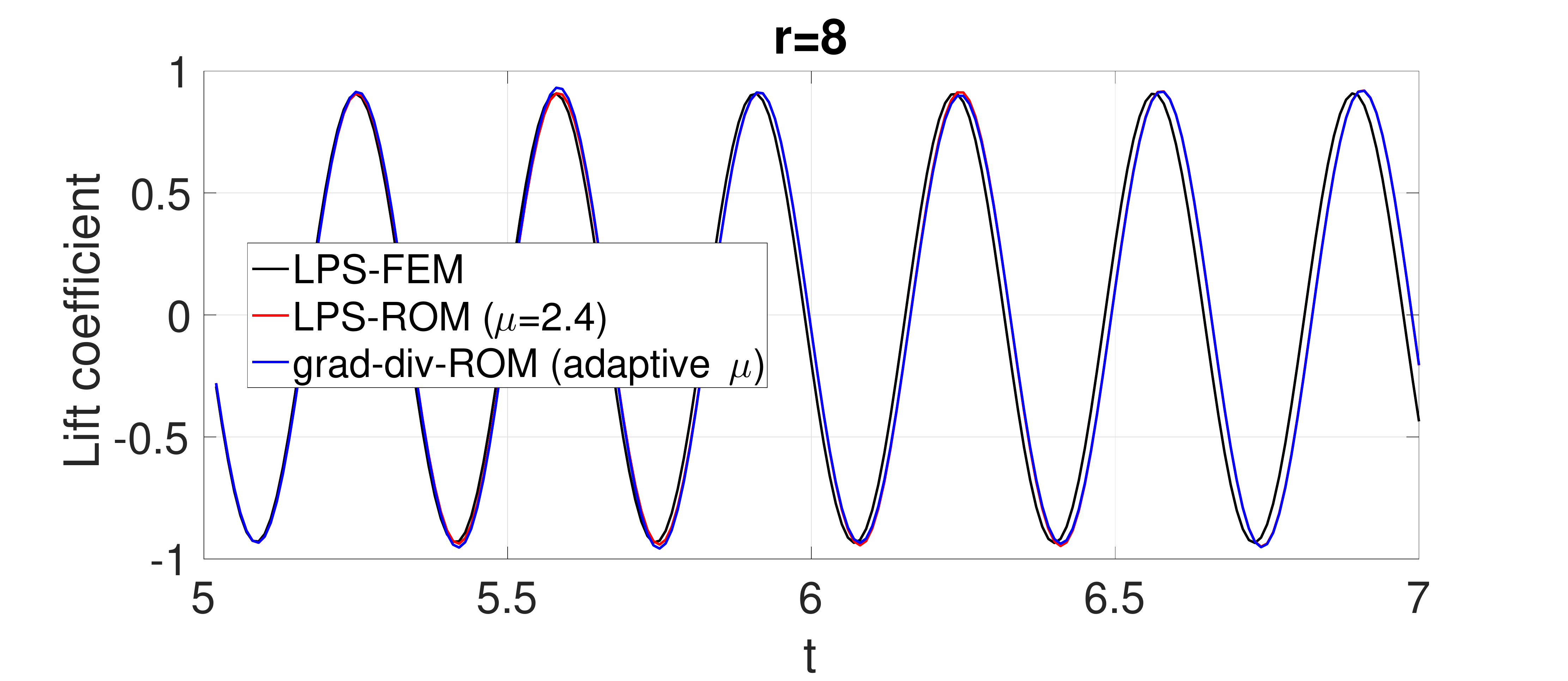}
\caption{Temporal evolution of drag coefficient (left) and lift coefficient (right) 
computed with LPS-ROM \eqref{eq:pod_method1} with constant and adaptive $\mu$ 
using $r=8$ velocity-pressure modes, and comparison with LPS-FEM \eqref{eq:gal_est}.}\label{fig:QOIPOD1LPS}
\end{center}
\end{figure}
%

\begin{figure}[htb]
\begin{center}
\includegraphics[width=2.5in]{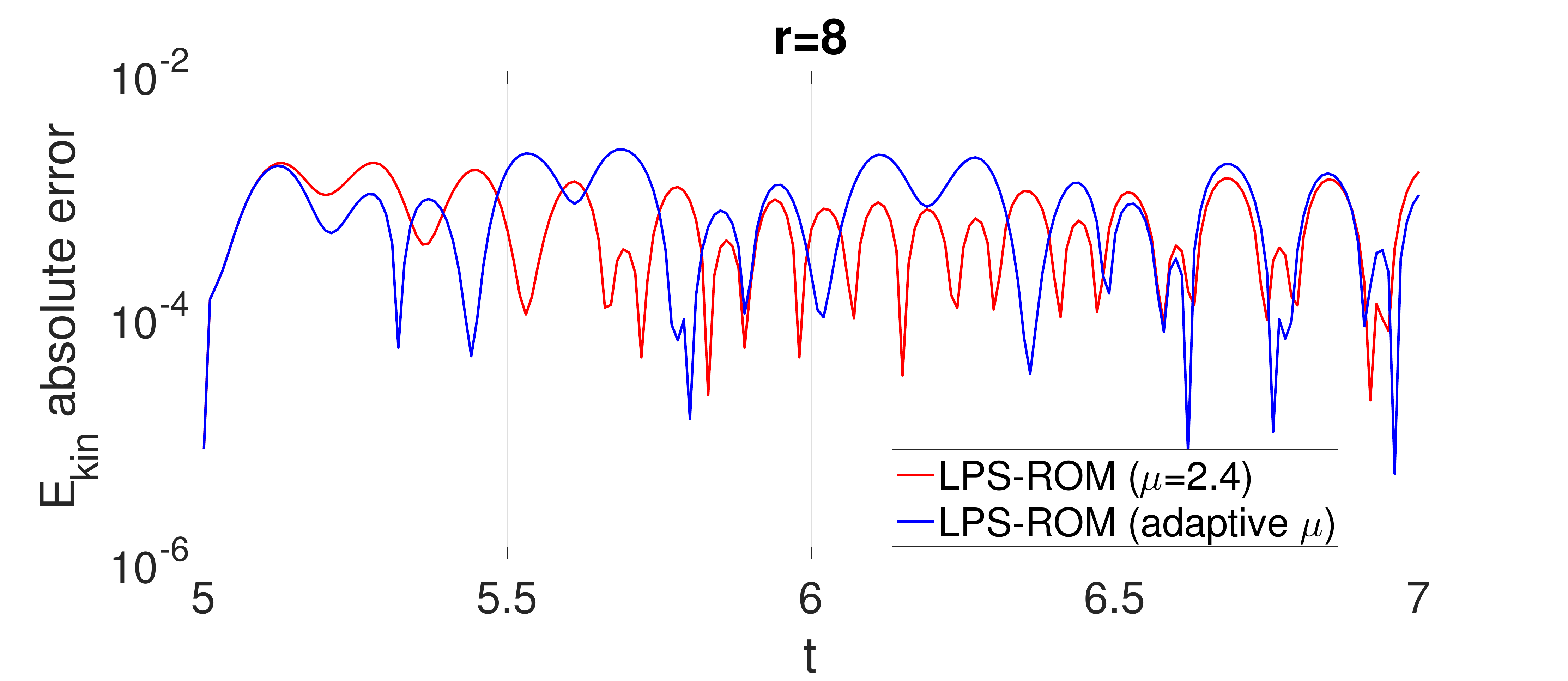}
\includegraphics[width=2.5in]{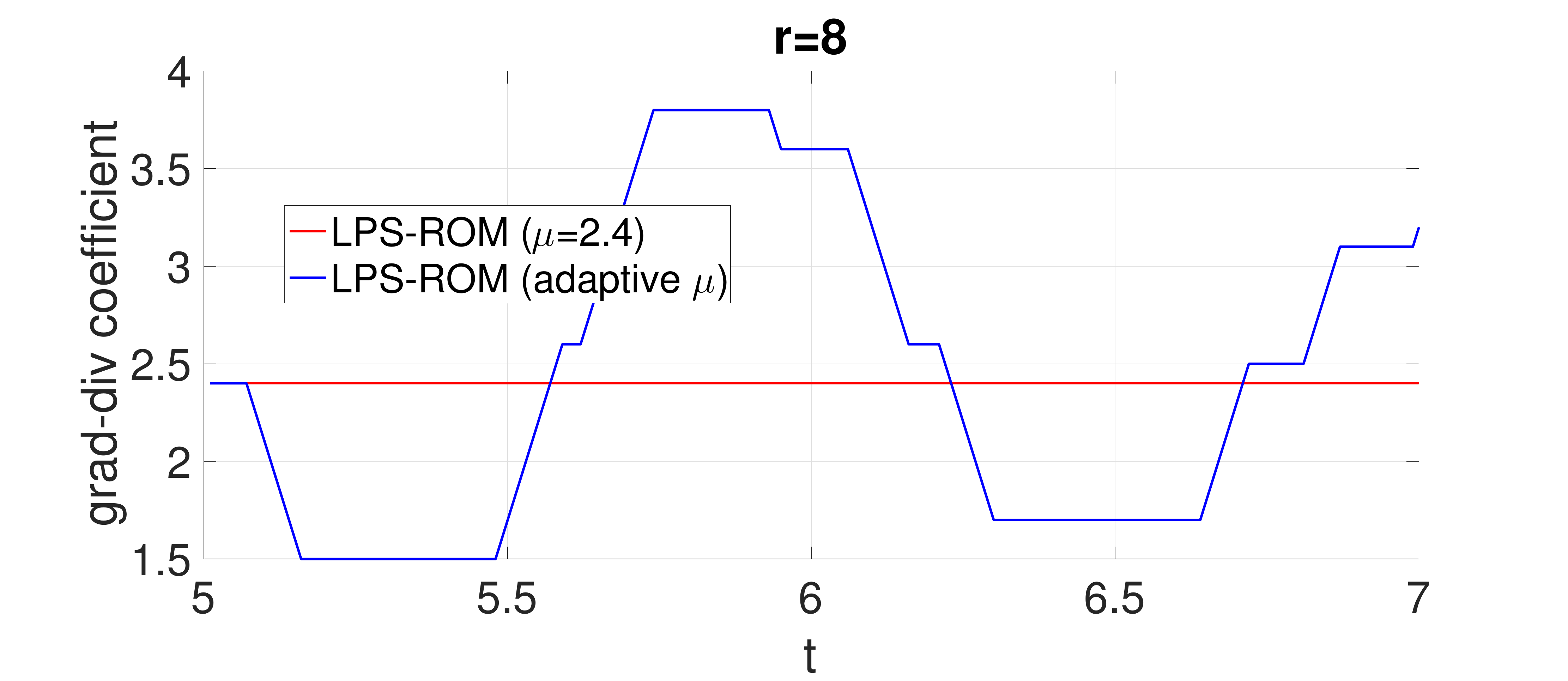}
\caption{Temporal evolution of absolute error in kinetic energy with respect to LPS-FEM \eqref{eq:gal_est} on the left and grad-div coefficient $\mu$ (constant and adaptive) on the right computed with LPS-ROM \eqref{eq:pod_method1} 
using $r=8$ velocity-pressure modes.}\label{fig:QOIPOD2LPS}
\end{center}
\end{figure}
%

%
To better assess on the one hand the behavior of the proposed LPS-ROM \eqref{eq:pod_method1} and partially illustrate on the other hand the theoretical convergence order predicted by the numerical analysis performed in Section \ref{sec:LPS-ROM}, we plot the discrete $\ell^2([5,7];L^2(\Om))$ squared errors in velocity and pressure with respect to the LPS-FEM \eqref{eq:gal_est} solution. In particular, in Figure \ref{fig:QOIPOD4LPS} we show the errors with respect to the LPS-FEM \eqref{eq:gal_est} solution computed with LPS-ROM \eqref{eq:pod_method1} with adaptive $\mu$ using $r$ velocity-pressure modes. The theoretical analysis proved that, for sufficiently small $h$ and $\Delta t$ as it is the case, the velocity error should scales as $\Lambda_r + Z_r = \sum_{k=r+1}^{d_v}\lambda_k+\sum_{k=r+1}^{d_p}\gamma_k$ (see error bound \eqref{eq:cota_finalLPS}), and this is quantitatively recovered in Figure \ref{fig:QOIPOD4LPS} for both reduced order velocity and pressure and a small number (up to $r=12$) of POD velocity-pressure modes. Following the hints given by the error bound \eqref{eq:cota_finalLPS}, for the current setup for which $S_2^v=\|S^v\|_2 = 1.94\cdot 10^2$ and $S_2^p=\|S^p\|_2 = 1.79\cdot 10^2$ (so that $h\,S_2^p={\cal O}(1)$), we actually found that $S_2^v \Lambda_r + Z_r$ is a good a priori error indicator for both reduced order velocity and pressure (at least for small $r$). However, for $r\geq 12$, we observe a flattening effect due to the fact that the time interval $[5,7]\,\rm{s}$ is already quite large with respect to the time period used to generate the POD basis, so that although we increase the number of POD modes (over $r=12$), we do not notice so much the error decrease in both velocity and pressure. 
Although for pressure we are not able in this case to theoretically prove error estimates in such a strong norm as for velocity (see Remark \ref{rm:PresErrEstLPS} and \cite{samuele_pod}), numerically we have recovered it. 

\begin{figure}[htb]
\begin{center}
\includegraphics[width=5.25in]{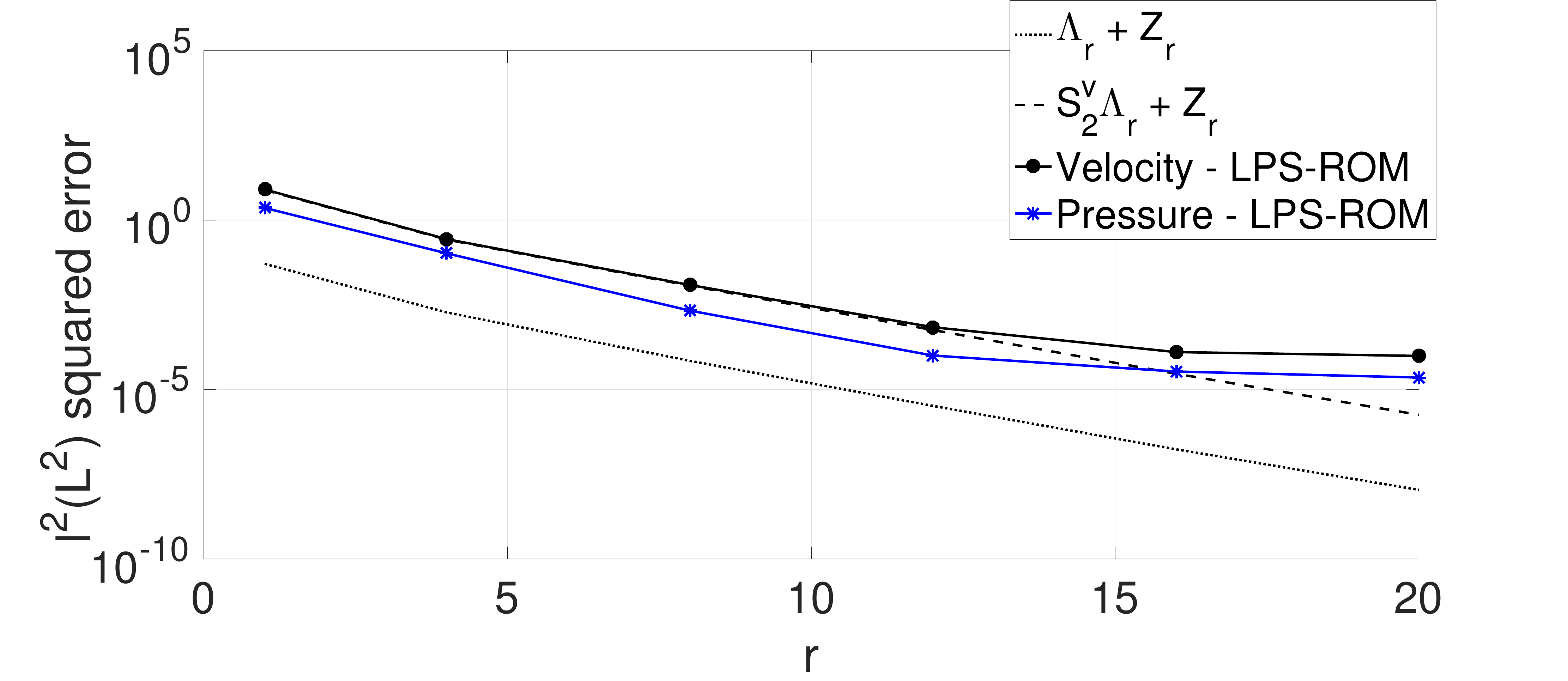}
\caption{Discrete $\ell^2(L^2)$ squared error in velocity and pressure with respect to LPS-FEM \eqref{eq:gal_est} computed with LPS-ROM \eqref{eq:pod_method1} with adaptive $\mu$ 
using $r$ velocity-pressure modes.}\label{fig:QOIPOD4LPS}
\end{center}
\end{figure}
%

We now report numerical results for the grad-div-ROM \eqref{eq:pod_method2}-\eqref{eq:pres}. Numerical results for drag and lift predictions using $r=8$ velocity-pressure modes are shown in Figure \ref{fig:QOIPOD1SUP}, where we display a comparison within grad-div-FEM \eqref{eq:gal_grad_div} and grad-div-ROM \eqref{eq:pod_method2}-\eqref{eq:pres} with constant $\mu=\bar{\mu}=3.7$ and adaptive $\mu$ (starting $\mu=\bar{\mu}=3.7$, $\mu_{min}=10^{-1}$, $F=5$, $\delta=10^{-1}$, $tol=10^{-3}$). In this case, the same number ($r=8$) of POD supremizers modes has been used to recover the online pressure. From this figure, we observe that the grad-div-ROM \eqref{eq:pod_method2}-\eqref{eq:pres} is less accurate when compared with the corresponding FOM (i.e., the grad-div-FEM \eqref{eq:gal_grad_div}) with respect to the LPS-ROM \eqref{eq:pod_method1} (see Figure \ref{fig:QOIPOD1LPS}). Indeed, the temporal evolution of the drag coefficient presents quite large differences (up to $12\%$) with respect to the grad-div-FEM \eqref{eq:gal_grad_div}, while remaining the lift coefficient temporal evolution rather close. Note that also in this case results are almost similar using constant and adaptive $\mu$ along the time interval $[5,7]\,\rm{s}$. In Figure \ref{fig:QOIPOD2SUP}, we show on the left the temporal evolution of absolute error in kinetic energy $|E_{kin,r}-E_{kin,h}|$, 
and on the right the corresponding temporal evolution of the adaptive grad-div coefficient $\mu$. Note that along the time interval $[5,7]\,\rm{s}$ the kinetic energy error levels are quite similar using constant and adaptive $\mu$, and in both cases are maintained below $2\cdot 10^{-3}$. We also notice that the values assumed by the online grad-div coefficient $\mu$ in order to better match the FOM energy are much larger than the offline grad-div-FEM coefficient $\mu=\mu_{K}$, whose maximum value is $2.76\cdot 10^{-2}\rm{m^2/s}$. 

\begin{figure}[htb]
\begin{center}
\includegraphics[width=2.5in]{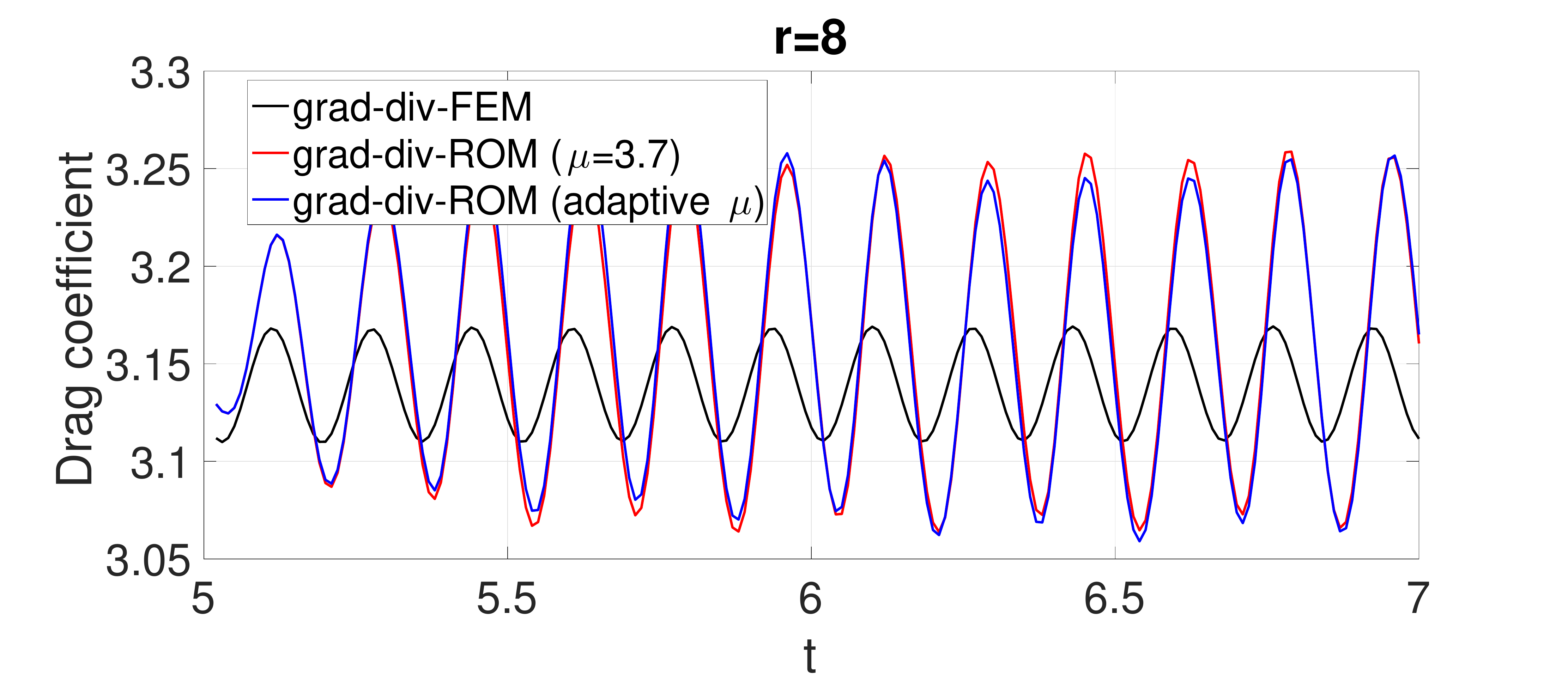}
\includegraphics[width=2.5in]{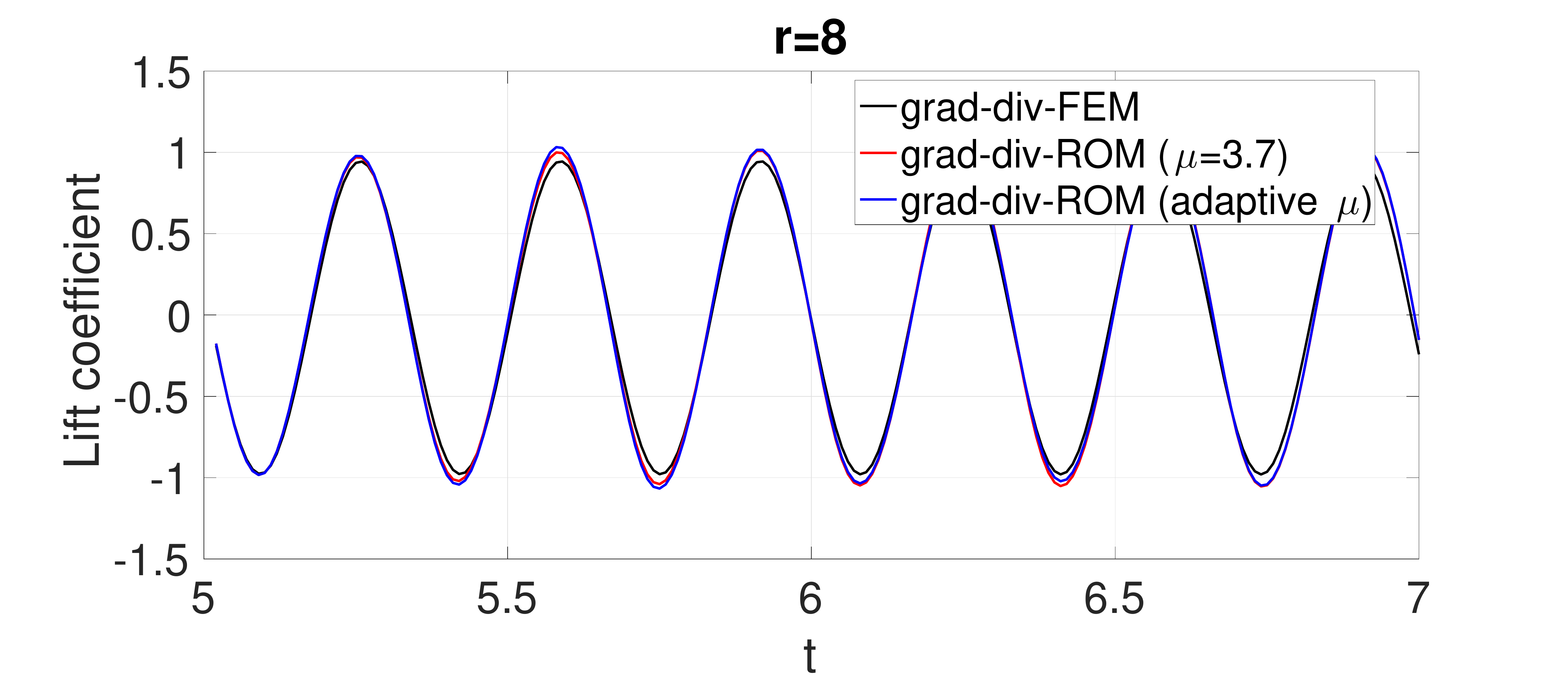}
\caption{Temporal evolution of drag coefficient (left) and lift coefficient (right) 
computed with grad-div-ROM \eqref{eq:pod_method2}-\eqref{eq:pres} with constant and adaptive $\mu$ using $r=8$ velocity-pressure modes, and comparison with grad-div-FEM \eqref{eq:gal_grad_div}.}\label{fig:QOIPOD1SUP}
\end{center}
\end{figure}

\begin{figure}[htb]
\begin{center}
\includegraphics[width=2.5in]{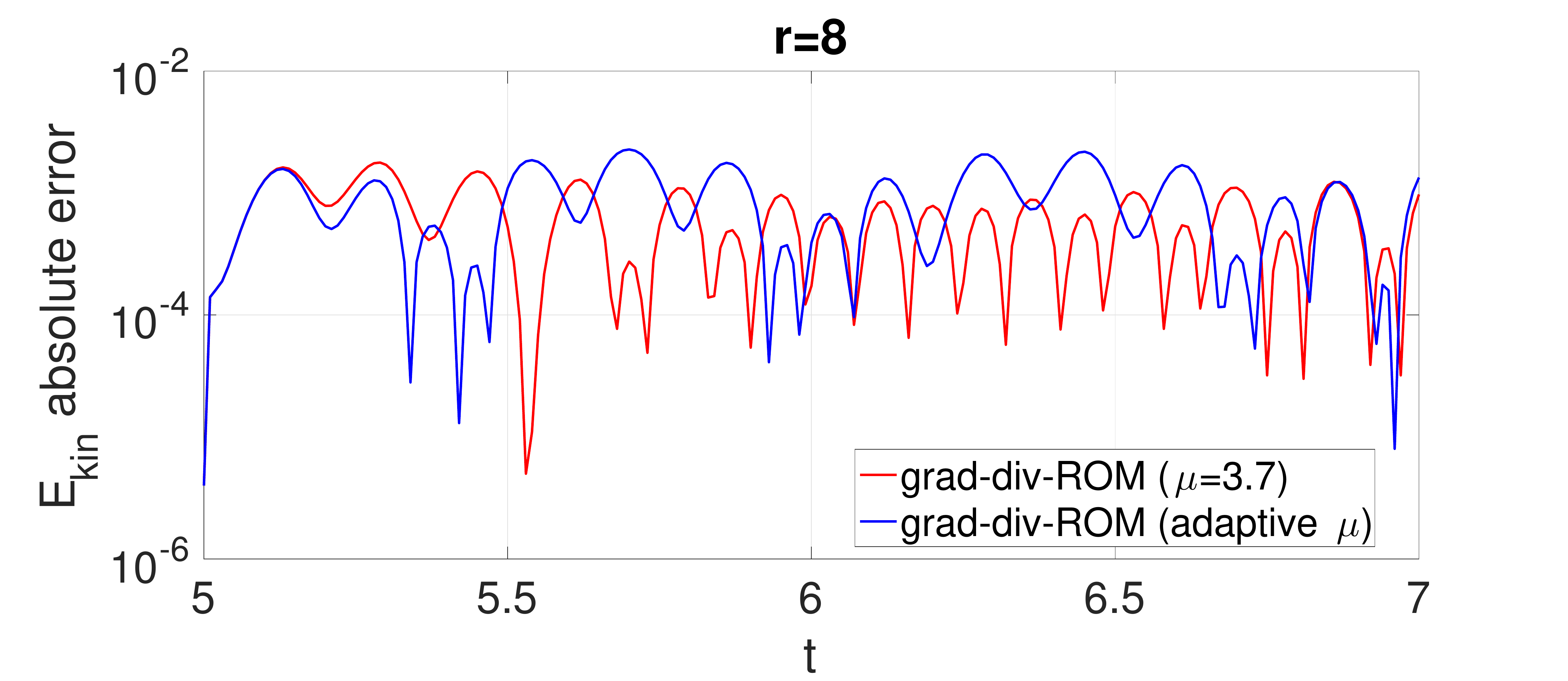}
\includegraphics[width=2.5in]{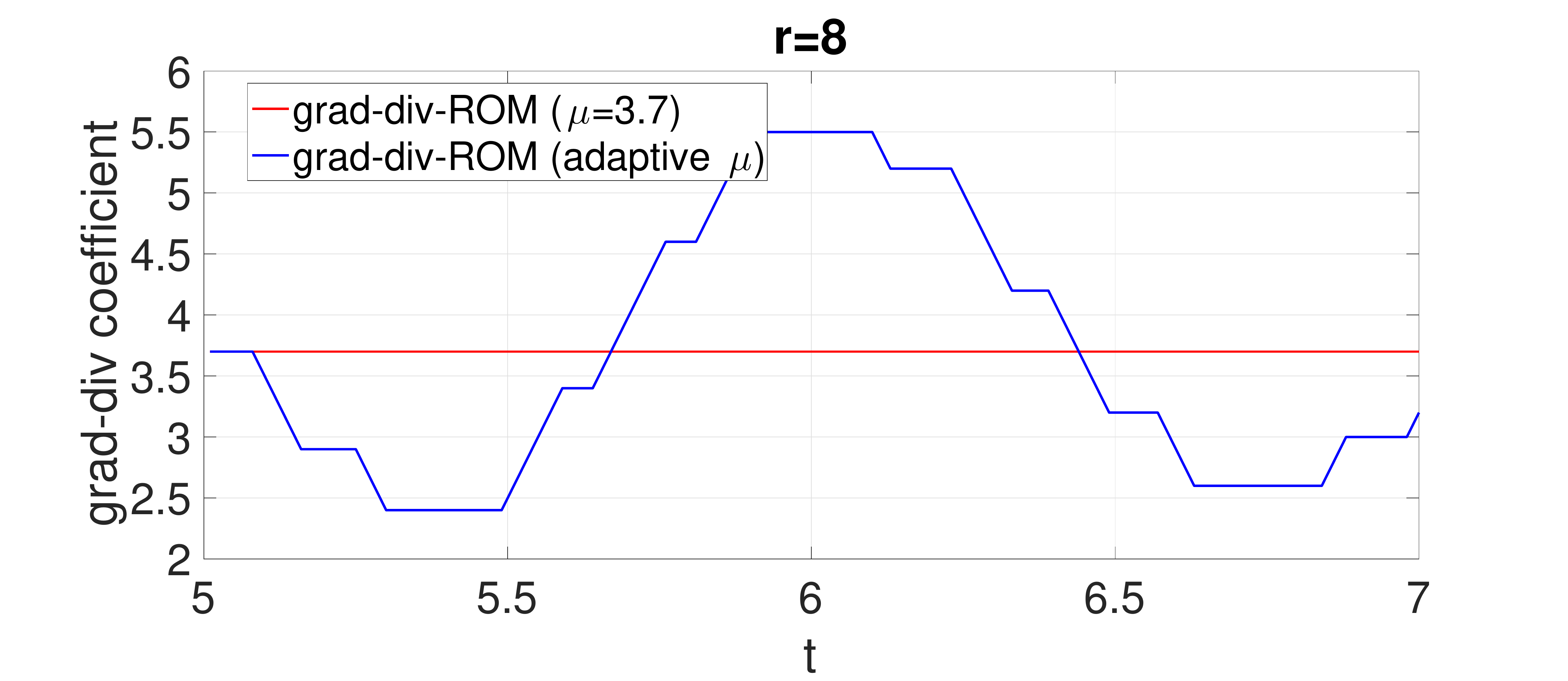}
\caption{Temporal evolution of absolute error in kinetic energy with respect to grad-div-FEM \eqref{eq:gal_grad_div} on the left and grad-div coefficient $\mu$ (constant and adaptive) on the right computed with grad-div-ROM \eqref{eq:pod_method2}-\eqref{eq:pres} using $r=8$ velocity-pressure modes.}\label{fig:QOIPOD2SUP}
\end{center}
\end{figure}
%

Also in this case, to better assess on the one hand the behavior of the proposed grad-div-ROM \eqref{eq:pod_method2}-\eqref{eq:pres} and partially illustrate on the other hand the theoretical convergence order predicted by the numerical analysis performed in Section \ref{sec:grad-div-ROM}, we plot the discrete $\ell^2([5,7];L^2(\Om))$ squared errors in velocity and pressure with respect to the grad-div-FEM \eqref{eq:gal_grad_div} solution. In particular, in Figures \ref{fig:QOIPOD4SUPv}-\ref{fig:QOIPOD4SUPp} we show the errors with respect to the grad-div-FEM \eqref{eq:gal_grad_div} solution computed with grad-div-ROM \eqref{eq:pod_method2}-\eqref{eq:pres} with adaptive $\mu$ using $r$ velocity-pressure modes. The theoretical analysis proved that, for sufficiently small $h$ and $\Delta t$ as it is the case, the velocity error should scales as $\Lambda_r$ (see error bound \eqref{eq:cota_finalSUPv}) and the pressure error as $\Lambda_r + Z_r$ (see error bound \eqref{eq:erpre4_ul}), and this is quantitatively recovered in Figures \ref{fig:QOIPOD4SUPv}-\ref{fig:QOIPOD4SUPp} for both reduced order velocity and pressure and a small number (up to $r=12$) of POD velocity-pressure (and supremizers) modes. Following the hints given by the error bounds \eqref{eq:cota_finalSUPv}, for the current setup for which $S_2^v=\|S^v\|_2 = 2.17\cdot 10^2$, we actually found that $S_2^v \Lambda_r$ is a good velocity error indicator, while following \eqref{eq:erpre4_ul}, we found that $\alpha C_r^{H^1} S_2^v \Lambda_r + Z_r$ is a good pressure error indicator (at least for small $r$). However, for $r\geq 12$, we observe again a flattening effect due to the fact that the time interval $[5,7]\,\rm{s}$ is already quite large with respect to the time period used to generate the POD basis, so that although we increase the number of POD modes (over $r=12$), we do not notice so much the error decrease in both velocity and pressure. 

\begin{figure}[htb]
\begin{center}
\includegraphics[width=5.25in]{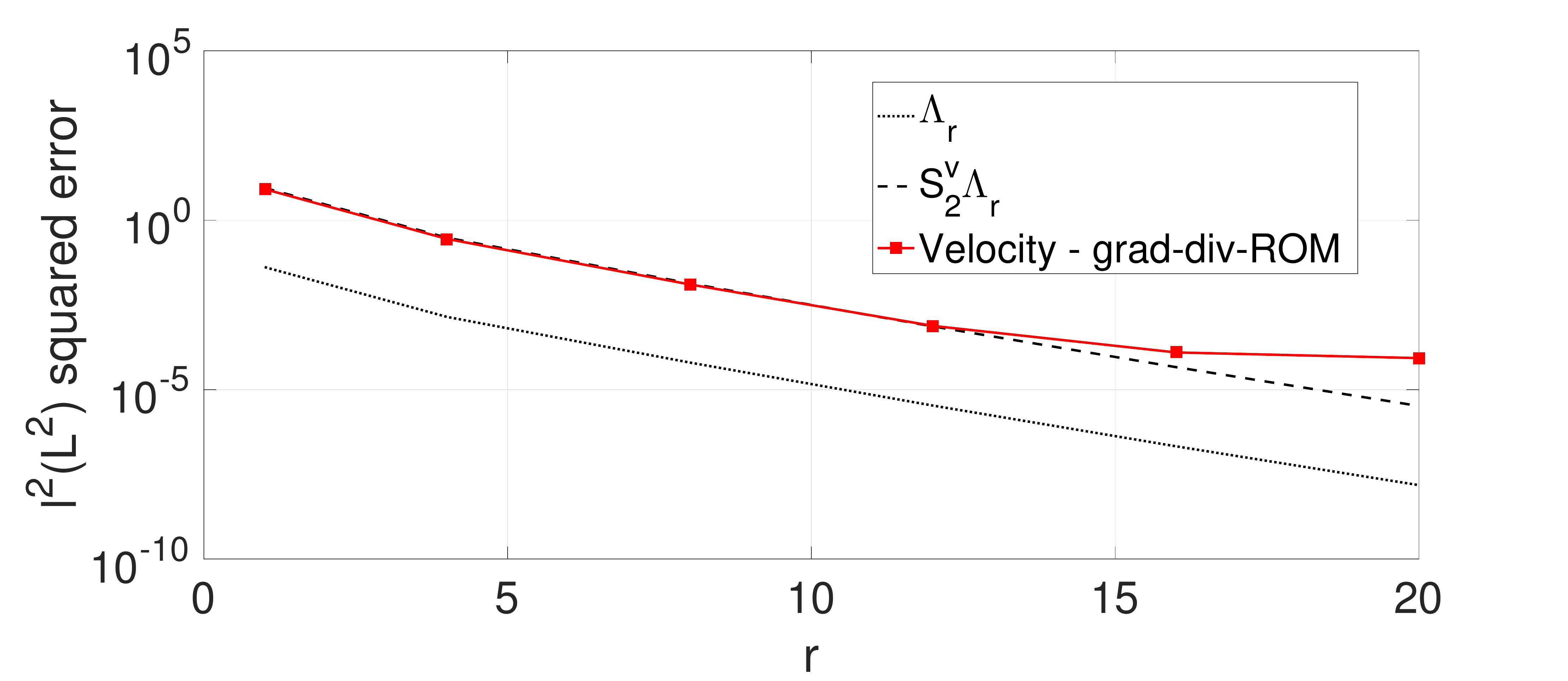}
\caption{Discrete $\ell^2(L^2)$ squared error in velocity with respect to grad-div-FEM \eqref{eq:gal_grad_div} computed with grad-div-ROM \eqref{eq:pod_method2}-\eqref{eq:pres} with adaptive $\mu$ 
using $r$ velocity-pressure modes.}\label{fig:QOIPOD4SUPv}
\end{center}
\end{figure}

\begin{figure}[htb]
\begin{center}
\includegraphics[width=5.25in]{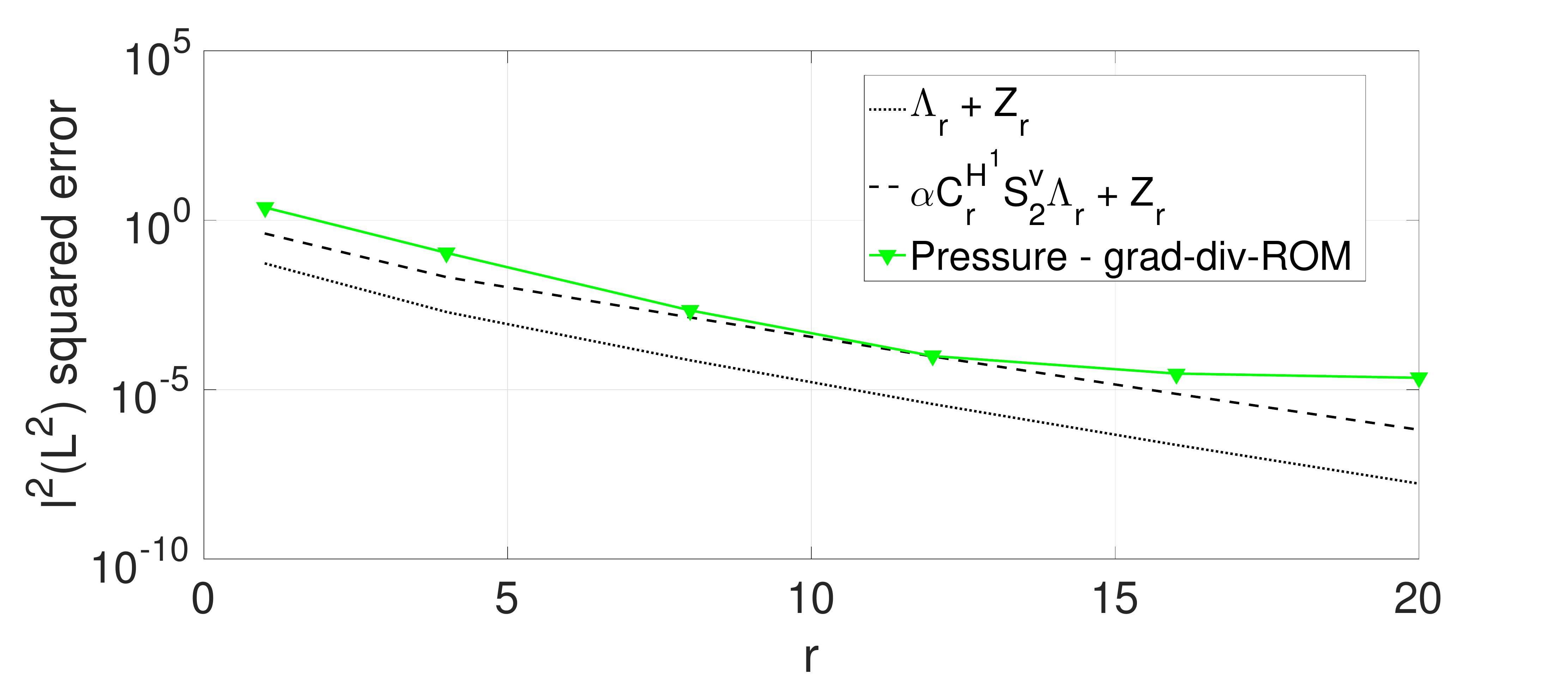}
\caption{Discrete $\ell^2(L^2)$ squared error in pressure with respect to grad-div-FEM \eqref{eq:gal_grad_div} computed with grad-div-ROM \eqref{eq:pod_method2}-\eqref{eq:pres} with adaptive $\mu$ 
using $r$ velocity-pressure modes.}\label{fig:QOIPOD4SUPp}
\end{center}
\end{figure}
%

\medskip

{\em Long time behavior.}
The aim of this section is to show how the strategy to adapt in time the online grad-div parameter can provide long time stability and accuracy, thus proving its robustness. To check the long time behavior, both the LPS-ROM \eqref{eq:pod_method1} and the grad-div-ROM \eqref{eq:pod_method2}-\eqref{eq:pres} are integrated in the time range $[5,20]\,\rm{s}$, which is forty-five times wider with respect to the time window used for the generation of the POD basis. 

The corresponding results for the LPS-ROM \eqref{eq:pod_method1} with constant and adaptive $\mu$ using $r = 8$ velocity-pressure modes are displayed in Figures \ref{fig:QOIPOD1longLPS}-\ref{fig:QOIPOD2longLPS}. In particular, in Figure \ref{fig:QOIPOD1longLPS} we monitor the temporal evolution of the drag and lift coefficients in the predictive time interval $[7,20]\,\rm{s}$. We observe that the LPS-ROM \eqref{eq:pod_method1} with adaptive $\mu$ remains stable and bounded, and gives reliable results for long time integration, since it rightly follows the trend given in the initial time range $[5,7]\,\rm{s}$ (see Figure \ref{fig:QOIPOD1LPS}), whereas maintaining a constant $\mu$ implies the oscillation amplitude (and of consequence the error) getting larger and larger as time increases. In Figure \ref{fig:QOIPOD2longLPS}, we show the long time evolution of kinetic energy (left) and the corresponding long time evolution of the adaptive grad-div coefficient $\mu$ (right). We observe that the adaptive strategy has a positive effect also on the long time energy evolution, causing the adaptive LPS-ROM \eqref{eq:pod_method1} energy to oscillate but remaining stable and rightly bounded over long time intervals, whereas maintaining a constant $\mu$ implies a significant inaccurate increase of the energy.

\begin{figure}[htb]
\begin{center}
\includegraphics[width=2.5in]{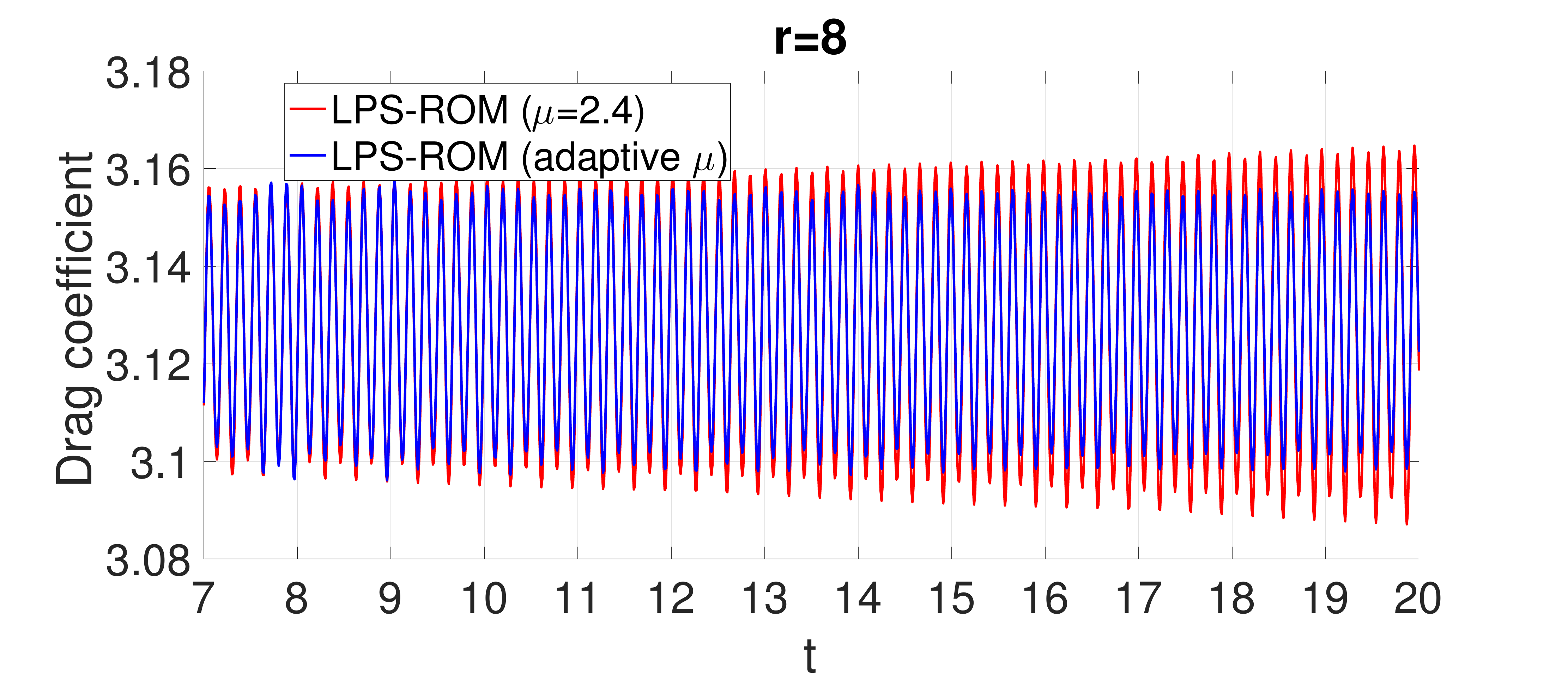}
\includegraphics[width=2.5in]{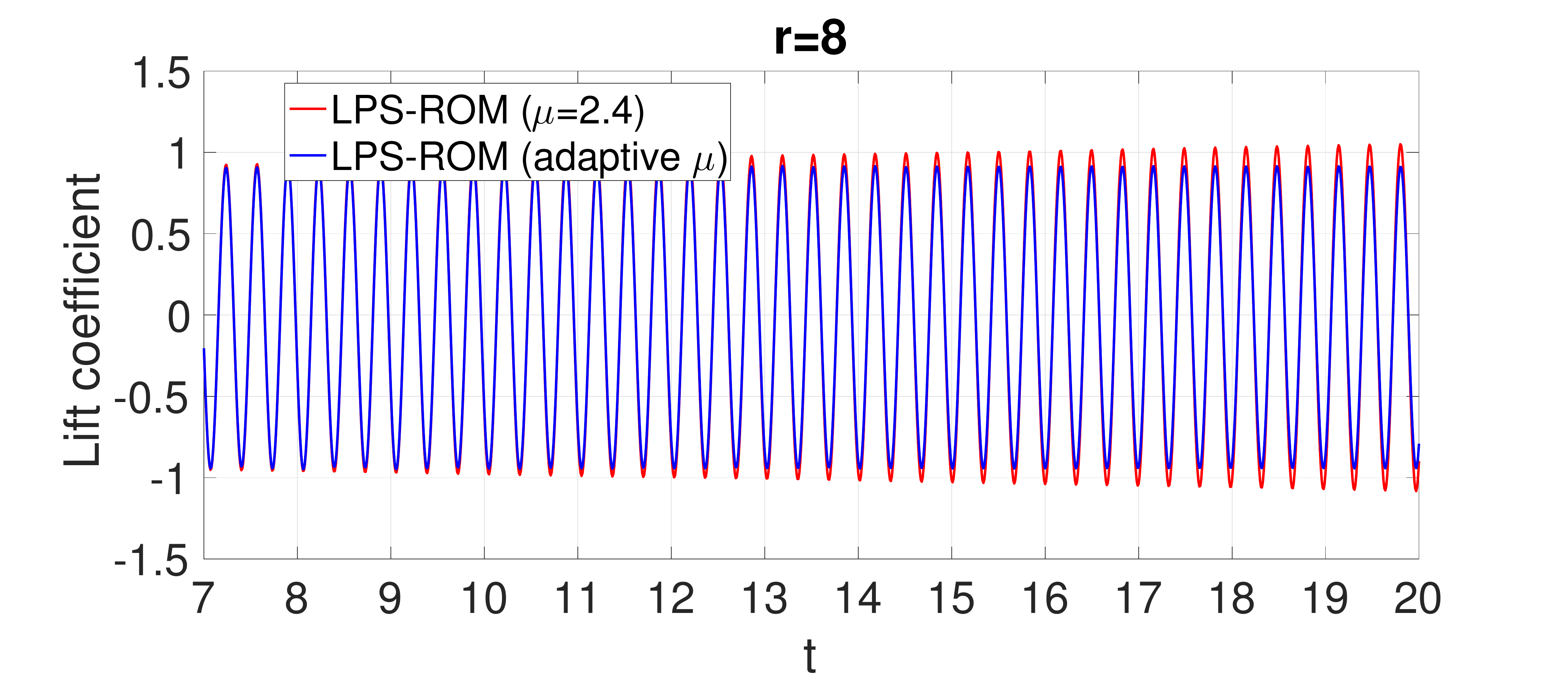}
\caption{Long time evolution of drag coefficient (left) and lift coefficient (right) 
computed with LPS-ROM \eqref{eq:pod_method1} with constant and adaptive $\mu$ 
using $r=8$ velocity-pressure modes.}\label{fig:QOIPOD1longLPS}
\end{center}
\end{figure}

\begin{figure}[htb]
\begin{center}
\includegraphics[width=2.5in]{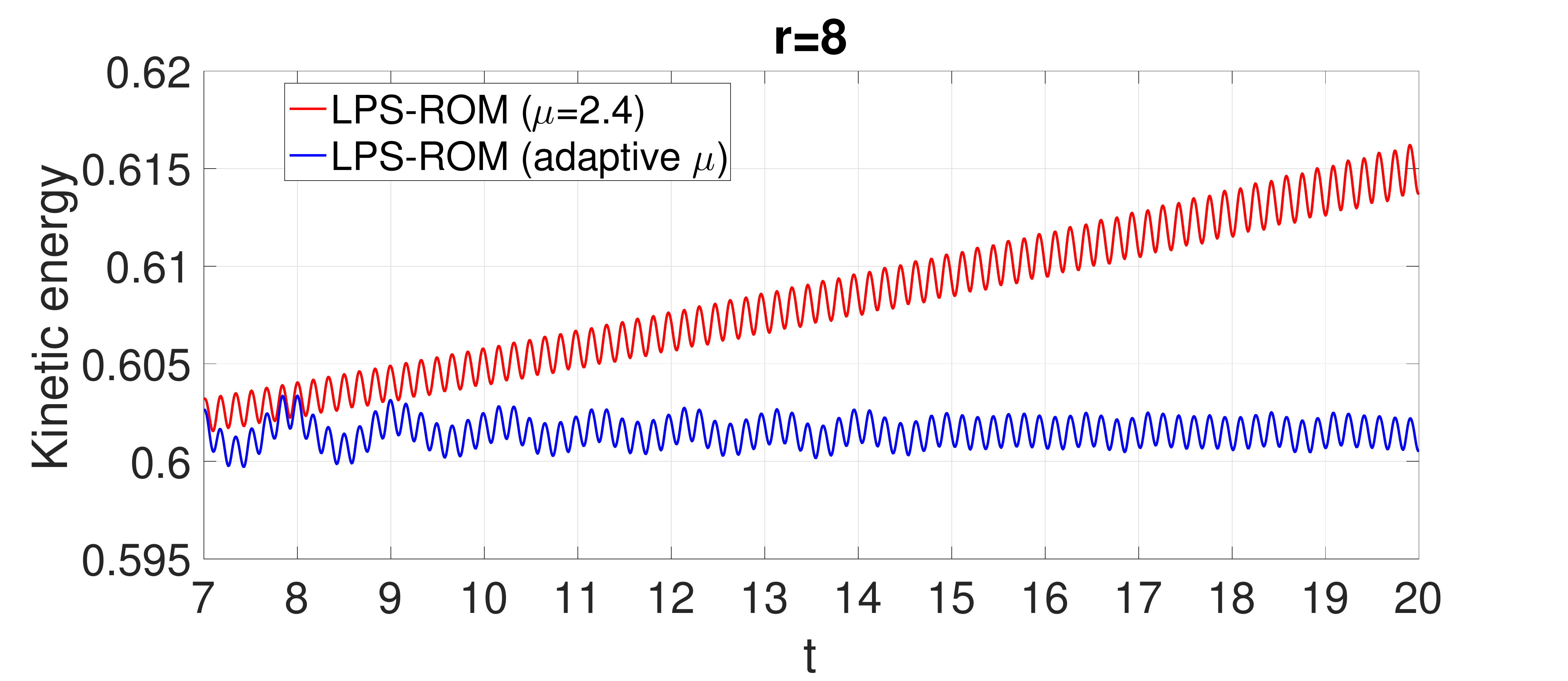}
\includegraphics[width=2.5in]{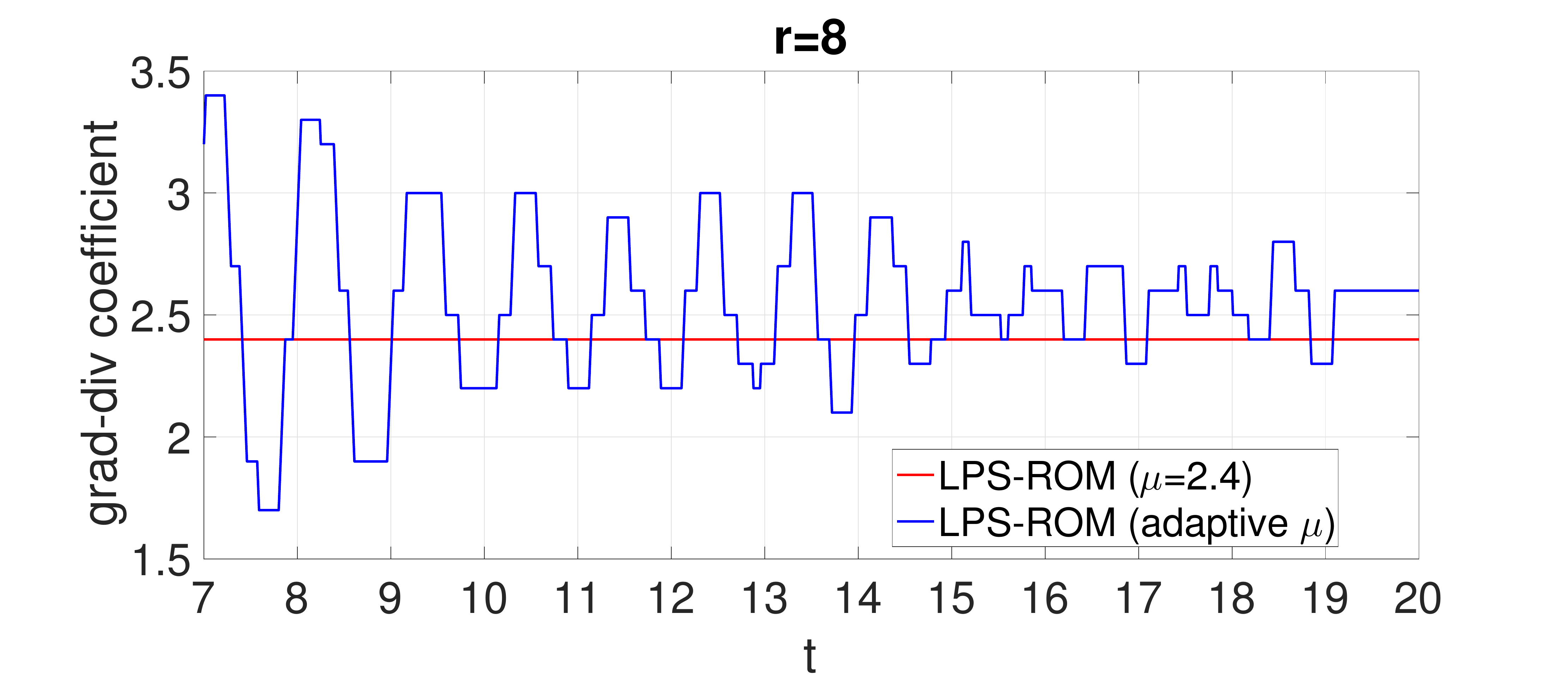}
\caption{Long time evolution of kinetic energy (left) and grad-div coefficient (right) computed with LPS-ROM \eqref{eq:pod_method1} with constant and adaptive $\mu$
using $r=8$ velocity-pressure modes.}\label{fig:QOIPOD2longLPS}
\end{center}
\end{figure}

The analogous results for the grad-div-ROM \eqref{eq:pod_method2}-\eqref{eq:pres} with constant and adaptive $\mu$ using $r = 8$ velocity-pressure (and supremizers) modes are displayed in Figures \ref{fig:QOIPOD1longSUP}-\ref{fig:QOIPOD2longSUP}. In particular, in Figure \ref{fig:QOIPOD1longSUP} we monitor the temporal evolution of the drag and lift coefficients in the predictive time interval $[7,20]\,\rm{s}$. Again, we observe that the grad-div-ROM \eqref{eq:pod_method2}-\eqref{eq:pres} with adaptive $\mu$ remains stable and bounded, and gives reliable results for long time integration, since it rightly follows the trend given in the initial time range $[5,7]\,\rm{s}$ (see Figure \ref{fig:QOIPOD1SUP}), whereas maintaining a constant $\mu$ implies the oscillation amplitude (and of consequence the error) getting larger and larger as time increases for the lift coefficient, and a totally inaccurate increasing beahvior for the drag coefficient. In Figure \ref{fig:QOIPOD2longSUP}, we show the long time evolution of kinetic energy (left) and the corresponding long time evolution of the adaptive grad-div coefficient $\mu$ (right). We observe again that the adaptive strategy has a positive effect also on the long time energy evolution, causing the adaptive grad-div-ROM \eqref{eq:pod_method2}-\eqref{eq:pres} energy to oscillate but remaining almost stable and rightly bounded over long time intervals, whereas maintaining a constant $\mu$ implies a significant inaccurate increase of the energy.

\begin{figure}[htb]
\begin{center}
\includegraphics[width=2.5in]{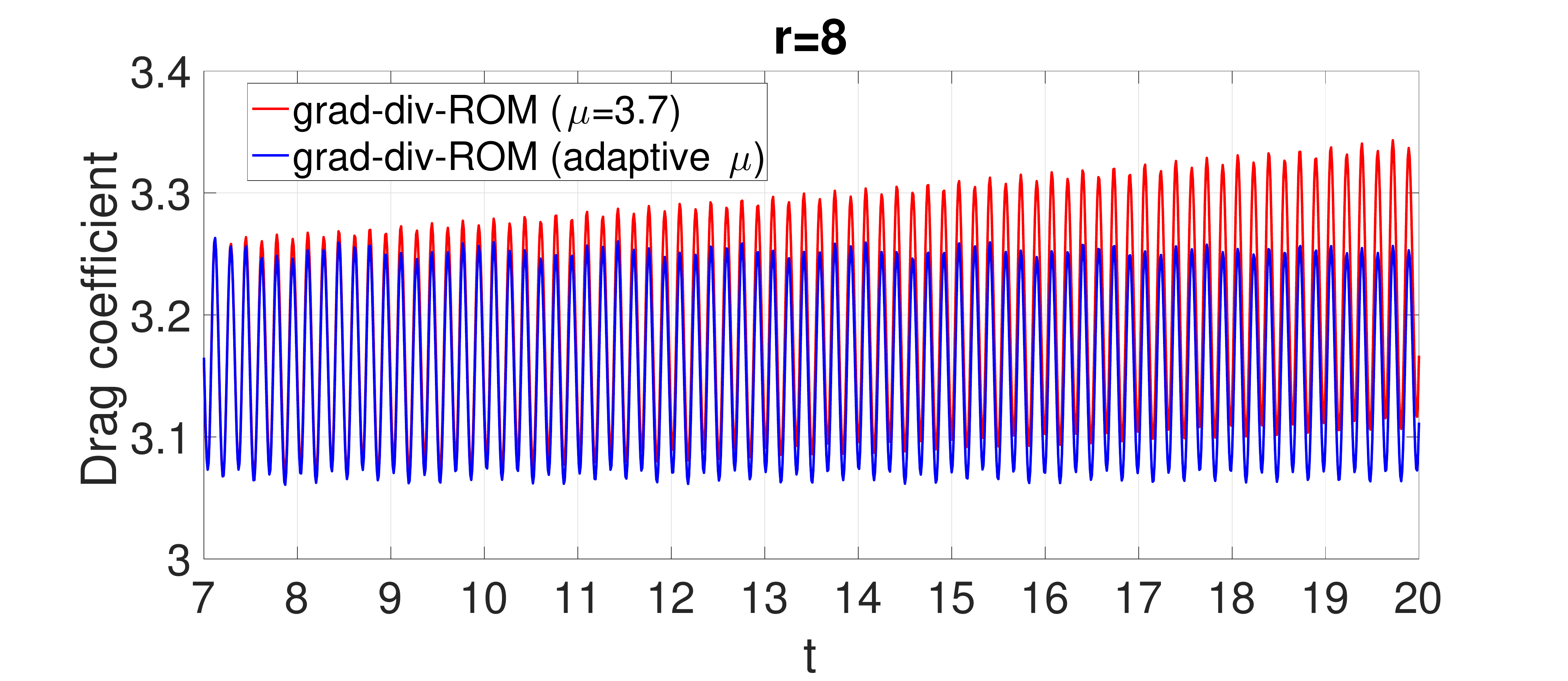}
\includegraphics[width=2.5in]{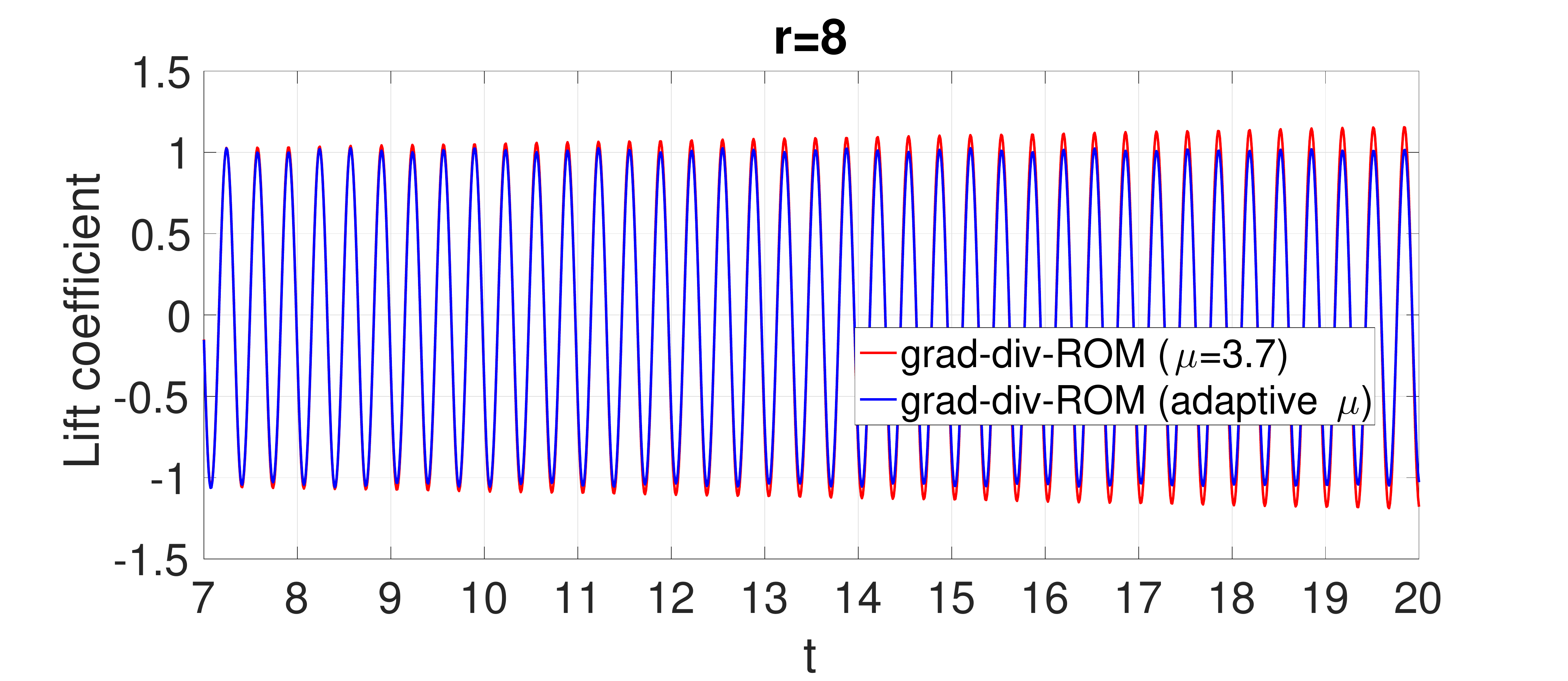}
\caption{Long time evolution of drag coefficient (left) and lift coefficient (right) 
computed with grad-div-ROM \eqref{eq:pod_method2}-\eqref{eq:pres} with constant and adaptive $\mu$ using $r=8$ velocity-pressure modes.}\label{fig:QOIPOD1longSUP}
\end{center}
\end{figure}

\begin{figure}[htb]
\begin{center}
\includegraphics[width=2.5in]{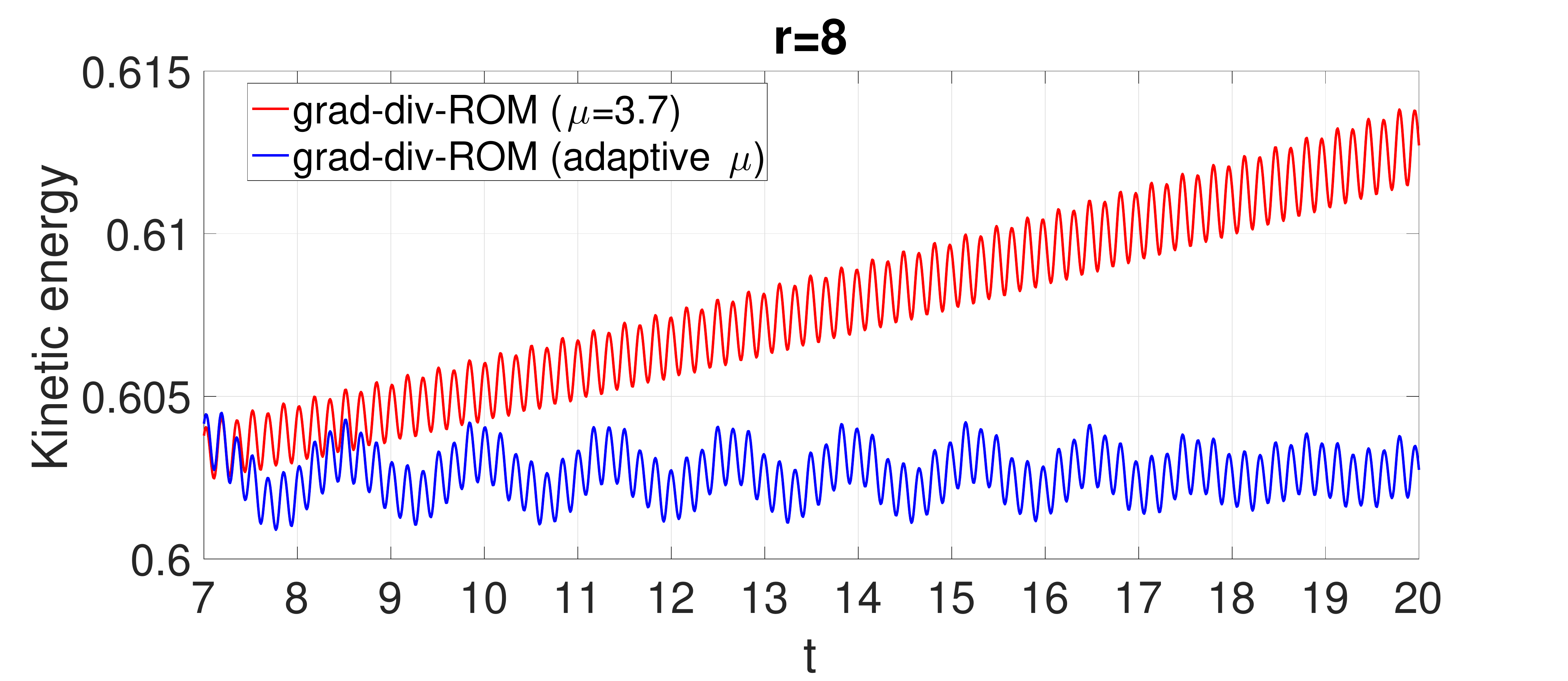}
\includegraphics[width=2.5in]{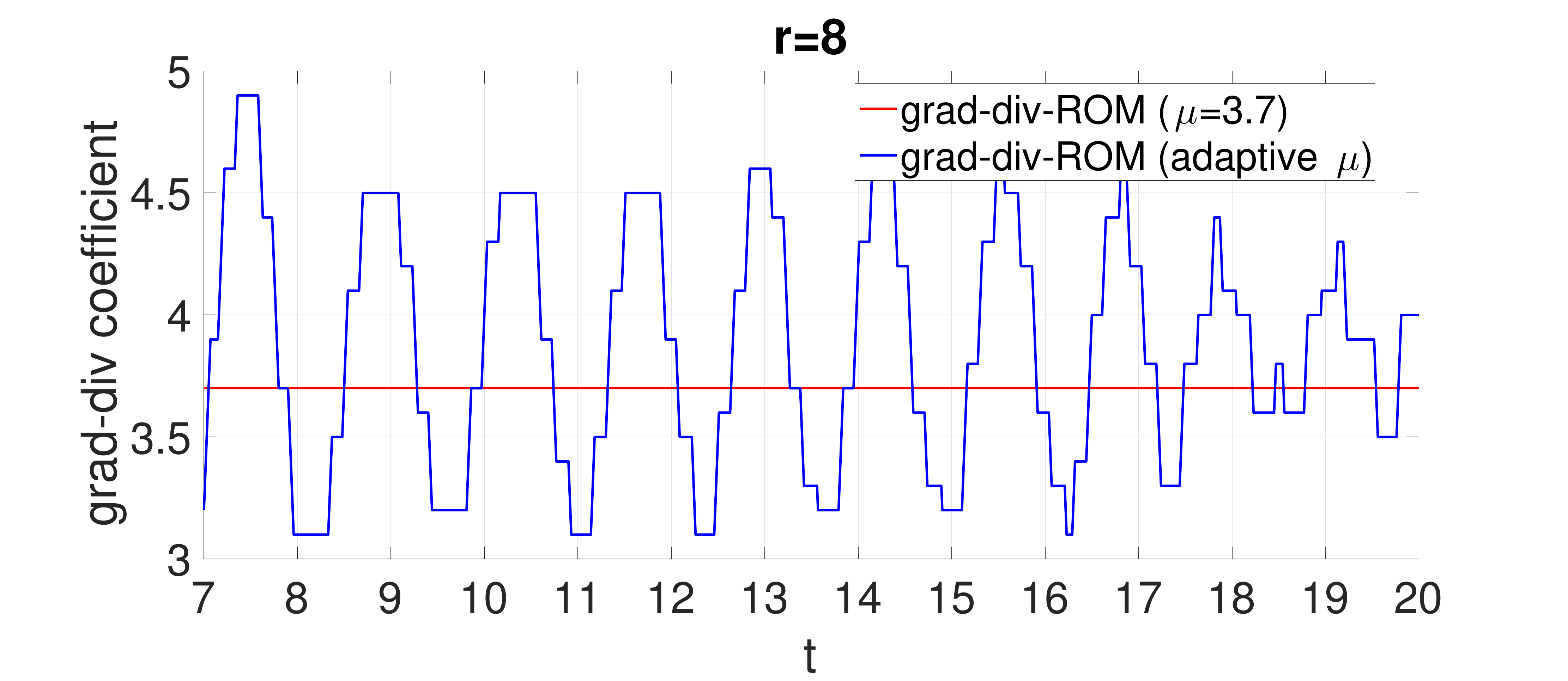}
\caption{Long time evolution of kinetic energy (left) and grad-div coefficient (right) computed with grad-div-ROM \eqref{eq:pod_method2}-\eqref{eq:pres} with constant and adaptive $\mu$ using $r=8$ velocity-pressure modes.}\label{fig:QOIPOD2longSUP}
\end{center}
\end{figure}


\section{Conclusions}\label{sec:Concl}

In this paper, POD stabilized reduced order methods for the numerical simulation of incompressible flows are proposed, analyzed and tested. In particular, we consider two approaches. In the first approach, called LPS-ROM, the standard discrete inf-sup condition is circumvented and POD modes are computed by a non inf-sup stable finite element method, called LPS-FEM \cite{nos_lps}. In the second approach, called grad-div-ROM, the standard discrete inf-sup condition is recovered through supremizer enrichment and POD modes are computed by an inf-sup stable finite element method, called grad-div-FEM \cite{NS_grad_div}. In both approaches, we build reduced order approximations for both velocity and pressure. In the first case, we consider as full order method a LPS finite element scheme with equal order interpolations, which stabilizes the gradient of both velocity and pressure. As for the corresponding reduced order POD method, we add the same kind of LPS for the gradient of both velocity and pressure than the FOM, together with online grad-div stabilization. In the second case, we consider as full order method an inf-sup stable Galerkin method with grad-div stabilization and for the corresponding reduced order POD method we also apply grad-div stabilization. In the latter case, since the velocity snapshots (and of consequence the POD velocity modes) are discretely divergence-free, the pressure can be removed from the formulation of the reduced order POD velocity, so that we use a momentum equation recovery approach to recover the online pressure from a velocity-only ROM, based on supremizer enrichment of the reduced velocity.

The main contribution of the present paper is the numerical analysis of the fully discrete LPS-ROM and grad-div-ROM applied to the unsteady incompressible Navier-Stokes equations, 
where rigorous error bounds with constants independent on inverse powers of the viscosity parameter are proved for both methods. To our knowledge, this is the first time this kind of sharp viscosity independent bounds are obtained for stabilized POD-ROM of incompressible flows. In this respect, the present study can be seen as an improvement of the numerical analysis performed in \cite{schneier} and \cite{samuele_pod}. 

Numerical experiments have been conducted to illustrate the compared performances of the proposed schemes and assess their accuracy and efficiency. Within this framework, we have also proposed an adaptive in time algorithm for the online grad-div parameter used both in the LPS-ROM and the grad-div-ROM, by adjusting dissipation arising from the online grad-div stabilization term in order to better match the FOM energy. In both cases, using a small equal number of POD velocity-pressure modes already provides comparable reliable approximations, close to the FOM results, and theoretical considerations suggested by the numerical analysis are recovered in practice. Actually, the performed analysis helped to find good a priori error indicators that, at least for small $r$ (of interest in practice), almost match the computed errors over predictive time intervals. Also, the adaptive in time algorithm for the online grad-div parameter proved to significantly improve the long time accuracy of both ROM. Although they provide similar global results, an increased accuracy can be observed for the LPS-ROM that better matches the corresponding FOM local quantities of interest such as the drag coefficient, especially. In terms of efficiency, note that the LPS-ROM circumvents the standard discrete inf-sup condition for the POD velocity-pressure spaces, whose fulfillment can be rather expensive and inefficient in realistic CFD applications \cite{Rozza15}, since it could require the application of the suprimizer enrichment strategy, used here for the grad-div-ROM, in a multi-parameter dependent setting. A cheaper and efficient alternative could be to use an approximate supremizer enrichment procedure \cite{Rozza15}, for which however it is not possible to rigorously show that the standard discrete inf-sup condition is satisfied and it is only possible to rely on heuristic criteria to check it in a post-processing stage. Finally, the velocity modes for the LPS-ROM do not have to be necessarily weakly divergence-free, which allows to use snapshots generated for instance with penalty or projection-based stabilized methods, as the LPS-FEM used in this paper. This is not the case of the grad-div-ROM considered here and, for instance, of ROM based on a pressure Poisson equation approach \cite{IliescuJohn14,schneier,RozzaStabile18} for which the velocity snapshots, and hence the POD velocity modes must be at least weakly divergence-free.

\bibliographystyle{abbrv}
\bibliography{references_sta,Biblio_STAB-POD-ROM_NSE}
\end{document}